\documentclass[journal]{IEEEtran}

\usepackage{amsmath}
\usepackage{graphicx}
\usepackage{amssymb}
\usepackage{epstopdf}
\usepackage{amsthm}
\usepackage{bm}
\usepackage{algorithm}
\usepackage{cite}
\usepackage[noend]{algpseudocode}
\usepackage{soul,color}
\usepackage{mathtools}
\usepackage{caption}
\usepackage{subcaption}
\usepackage{array}
\usepackage{enumitem}

\DeclarePairedDelimiter{\norm}{\lVert}{\rVert}
\DeclarePairedDelimiter{\abs}{\lvert}{\rvert}

\def\S{{\cal S}}
\def\P{{\cal P}}

\def\R{{\mathbb R}}

\def\C{{\cal C}}

\def\N{{\cal N}}

\newcommand{\topnew}{\top \hspace{-0.05cm}}
\newcommand\numberthis{\addtocounter{equation}{1}\tag{\theequation}}

\algnewcommand{\IfThenElse}[3]{
  \State \algorithmicif\ #1\ \algorithmicthen\ #2\ \algorithmicelse\ #3}

\DeclareMathOperator*{\argmin}{argmin}

\DeclareMathOperator*{\vect}{vec}

\DeclareMathOperator*{\diag}{diag}

\newtheorem{theorem}{Theorem}

\newtheorem{proposition}{Proposition}
\newtheorem{lemma}{Lemma}
\newtheorem{remark}{Remark}

\newtheorem{definition}{Definition}

\makeatletter
\newcommand\fs@betterruled{%
 \def\@fs@cfont{\bfseries}\let\@fs@capt\floatc@ruled
 \def\@fs@pre{\vspace*{5pt}\hrule height.8pt depth0pt \kern2pt}%
 \def\@fs@post{\kern2pt\hrule\relax}%
 \def\@fs@mid{\kern2pt\hrule\kern2pt}%
 \let\@fs@iftopcapt\iftrue}
\floatstyle{betterruled}
\restylefloat{algorithm}
\makeatother

\begin{document}
%
\title{On Local Linear Convergence of Projected Gradient Descent for Unit-Modulus Least Squares}
%
%
%

\author{Trung~Vu,~\IEEEmembership{Graduate Student Member,~IEEE,}~and~Raviv Raich,~\IEEEmembership{Senior Member,~IEEE},~and~Xiao Fu,~\IEEEmembership{Senior Member,~IEEE}
\thanks{Trung Vu, Raviv Raich, and Xiao Fu are with the School of Electrical Engineering and Computer Science, Oregon State University, Corvallis, OR 97331, USA.}
\thanks{E-mails: vutru@oregonstate.edu; raich@oregonstate.edu; and xiao.fu@oregonstate.edu}
\thanks{Manuscript received 2022.}}

%
%

\markboth{Preprint}%
{Trung \MakeLowercase{\textit{et al.}}: On Local Linear Convergence of Projected Gradient Descent for Unit-Modulus Least Squares}
%



\maketitle

\begin{abstract}
The unit-modulus least squares (UMLS) problem has a wide spectrum of applications in signal processing, e.g., phase-only beamforming, phase retrieval, radar code design, and sensor network localization. Scalable first-order methods such as projected gradient descent (PGD) have recently been studied as a simple yet efficient approach to solving the UMLS problem. Existing results on the convergence of PGD for UMLS often focus on global convergence to stationary points. As a non-convex problem, only a sublinear convergence rate has been established. However, these results do not explain the fast convergence of PGD frequently observed in practice. This manuscript presents a novel analysis of convergence of PGD for UMLS, justifying the linear convergence behavior of the algorithm near the solution. By exploiting the local structure of the objective function and the constraint set, we establish an exact expression for the convergence rate and characterize the conditions for linear convergence. Simulations show that our theoretical analysis corroborates numerical examples. Furthermore, variants of PGD with adaptive step sizes are proposed based on the new insight revealed in our convergence analysis. The variants show substantial acceleration in practice.
\end{abstract}


%
\IEEEpeerreviewmaketitle


\section{Introduction}
%
%
%
%

\IEEEPARstart{U}{nit}-modulus least squares (UMLS) is formulated as the following optimization problem:
\begin{align*} 
    \min_{\bm w \in \mathbb{C}^N} \quad & \frac{1}{2} \norm{\bm \Phi \bm w - \bm h}^2 \\
    \text{ s.t. } \quad & \abs{w_i}^2=1 \text{ for } i=1,\ldots,N , \numberthis \label{prob:orginal}
\end{align*}
where $\bm \Phi \in \mathbb{C}^{M \times N}$ and $\bm h \in \mathbb{C}^M$. This problem arises in numerous machine learning and signal processing applications including, but not limited to, phase-only beamforming \cite{lu2013novel,tranter2017fast}, phase retrieval \cite{candes2013phaselift,waldspurger2015phase}, radar code design \cite{thompson1976adaptation,soltanalian2014designing}, and sensor network localization \cite{fu2012complex}.
For instance, in phase-only beamforming applications, the goal is to design a weight vector $\bm w$ associated with $N$ antennas so that it retains the power of each antenna and enhances reception of the signals from certain directions while mitigating interference located at other directions. For a uniform linear antenna array, $\bm \Phi$ can be the steering vector matrix with a Vandermonde structure.

It is well-known that UMLS is a non-convex NP-hard problem \cite{zhang2006complex}. One traditional approach to this problem is semi-definite relaxation (SDR). In \cite{luo2010semidefinite}, Luo~\textit{et.~al.~} recast (\ref{prob:orginal}) as a quadratically constrained quadratic programming (QCQP) problem and then lifted it to an $N^2$-dimensional problem with a rank-$1$ constraint. By dropping the non-convex rank constraint, the resulting problem is convex and can be solved via interior point methods. The major disadvantage of SDR is the high computational complexity ($O(N^7)$ flops and $O(N^2)$ memory units), which is not suitable for large-scale problems in modern applications.
Another approach that has recently been proposed by Tranter~\textit{et.~al.} \cite{tranter2017fast} is projected gradient descent (PGD). Since the projection onto the unit-modulus manifold is simple and low-cost, PGD was shown to be efficient in large-scale settings. Notably, the authors in \cite{tranter2017fast} showed that despite the lack of convexity, the algorithm converges globally to a set of stationary points of (\ref{prob:orginal}) and the rate of convergence is at least sublinear. 

Motivated by Tranter's result, this manuscript provides an in-depth convergence analysis of PGD for UMLS. 
First, we observe in practice that the algorithm frequently exhibits linear convergence near a local minimum of the problem. This is significantly faster than the sublinear convergence proven in \cite{tranter2017fast}. 
Second, the bounding technique in \cite{tranter2017fast} is rather conservative since it focuses on global characterization yet ignores the local structure of the problem around the solution.
In particular, while UMLS is not a globally convex problem, it can still possess a benign geometry around a local minimum. In such scenario, one can expect that PGD will converge linearly to the local minimum similar to gradient descent for unconstrained minimization of a smooth and strongly convex function \cite{nesterov2003introductory}.
With this intuition, our goal here is to provide an analytical framework to uncover the fast linear convergence behavior of PGD near a local minimum of the UMLS problem.\footnote{Preliminary aspects of this work appeared in an earlier conference version \cite{vu2019convergence}, where we study the local convergence of PGD for minimizing a quadratic over the unit sphere. When $N=1$, the UMLS problem and the spherically constrained least squares problem coincide. For $N>1$, UMLS introduces a more complex constraint set in the form of the cross product of multiple spherical constrains.}
By exploiting the structure of the problem near local minima, we are able to identify the sufficient conditions for local linear convergence of PGD with a fixed step size and obtain an exact expression of the convergence rate. 
In addition, we establish the region of convergence in which initializing the algorithm is guaranteed to converge to the desired local minimum. The theoretical rate predicts accurately the empirical convergence rate in our numerical simulation. Finally, in practical applications where prior knowledge of the solution is not available, we propose two adaptive-step-size variants of PGD that require the same iteration complexity while offering faster linear convergence compared to the optimal fixed step size in theory.

The rest of the paper is organized as follows.
Section~\ref{sec:ps} presents the real-valued formulation of the UMLS problem that is considered in this paper and the PGD algorithm for solving this problem.
Section~\ref{sec:prel} summarizes existing results on the convergence of PGD for UMLS in the literature, highlighting the fundamental similarity between the UMLS problem and the spherically constrained least squares problem.
Our convergence analysis is presented in Section~\ref{sec:analysis}, including solution properties, algorithm properties, and linear convergence properties.
In Section~\ref{sec:adaptive}, we propose two variants of PGD for UMLS that use adaptive step size schemes to effectively obtain fast linear convergence without prior knowledge of the solution.
Finally, in Section~\ref{sec:exp}, we perform numerical experiments to verify our theoretical analysis.

\section{Problem Statement}
\label{sec:ps}

In this section, we introduce fundamental concepts in formulating the UMLS problem as a standard constrained least squares optimization and the PGD algorithm for solving it.

\subsection{Notation} 
\label{sec:notation}

Throughout the paper, we use the notations $\norm{\cdot}_F$ and $\norm{\cdot}_2$ to denote the Frobenius norm and the spectral norm of a matrix, respectively. Additionally, $\norm{\cdot}$ is used on a vector to denote the Euclidean norm. Boldfaced symbols are reserved for vectors and matrices. The $t \times t$ identity matrix is denoted by $\bm I_t$. The $t$-dimensional vector of all zeros and the $t$-dimensional vector of all ones are denoted by $\bm 0_t$ and $\bm 1_t$, respectively. The notations $\otimes$ denotes the Kronecker product between two matrices and $\vect(\cdot)$ denotes the vectorization of a matrix by stacking its columns on top of one another. For a complex number $z$, $\Re(z)$ and $\Im(z)$ denote the real and imaginary parts of $z$, respectively. Given an $n$-dimensional vector $\bm x$, $x_i$ denotes its $i$th element and $\diag(\bm x)$ denotes the $n \times n$ diagonal matrix with the corresponding diagonal entries $x_1,\ldots,x_n$.
Given a matrix $\bm X \in \R^{m \times n}$, the $i$th largest eigenvalue and the $i$th largest singular value of $\bm X$ are denoted by $\lambda_i(\bm X)$ and $\sigma_i(\bm X)$, respectively. The spectral radius of $\bm X$ is defined as $\rho(\bm X) = \max_{i} \abs{\lambda_i(\bm X)}$ and is less than or equal to the spectral norm, i.e., $\rho(\bm X) \leq \norm{\bm X}_2$ \cite{meyer2000matrix}. If $\bm X$ is square and invertible, the condition number of $\bm X$ is defined as $\kappa(\bm X) = \sigma_1(\bm X)/\sigma_n(\bm X)$. Finally, we use $\bm X \succ 0$ to indicate the matrix $\bm X$ is positive definite (PD) and $\bm X \succeq 0$ to indicate the matrix $\bm X$ is positive semi-definite (PSD).


\subsection{Real-valued Formulation of the UMLS Problem}

For the convenience of analysis, we consider the following real-valued reparametrization of (\ref{prob:orginal}):
\begin{align*}
    \min_{\bm x \in \R^{2N}} \quad &\frac{1}{2} \norm{\bm A \bm x - \bm b}^2 \\
    \text{ s.t. } \quad & x_{2i-1}^2 + x_{2i}^2 = 1 \text{ for } i=1,\ldots,N , \numberthis \label{prob:umls0}
\end{align*}
where $\bm A \in \R^{2M \times 2N}$ is partitioned into $2 \times 2$ blocks of form
\begin{align} \label{equ:A_ij}
    \tilde{\bm A}_{ij} = \begin{bmatrix} \Re(\Phi_{ij}) & -\Im(\Phi_{ij}) \\ \Im(\Phi_{ij}) & \Re(\Phi_{ij}) \end{bmatrix}, 
\end{align}
for $i=1,\ldots,M$ and $j=1,\ldots,N$.
In addition, $\bm x = [\Re(w_1),\Im(w_1),\ldots,\Re(w_N),\Im(w_N)]^\topnew$ and $\bm b = [\Re(h_1),\Im(h_1),\ldots,\Re(h_M),\Im(h_M)]^\topnew$ are real-valued vectors.
Next, we introduce the concepts of the 2-selection operator that selects the $i$th coordinate pair of a $2N$-dimensional vector. Since the unit-modulus constraint involves every pair of coordinates of $\bm x$, this operator allows us to simplify the representation of our result throughout the rest of the paper:
\begin{definition} \label{def:Si}
For each $i=1,\ldots,N$, the $i$th \textbf{2-selection operator} is defined by $\bm S_i: \R^{2N} \to \R^2$ such that
\begin{align*}
    \bm S_i(\bm x) = \begin{bmatrix} x_{2i-1} \\ x_{2i} \end{bmatrix} ,
\end{align*}
where ${\bm x}=[x_1,x_2,\ldots,x_{2N}]^\topnew$.
\end{definition}
\noindent It is noteworthy that the 2-selection operators is linear.
Using this operator, we can represent any vector $\bm x \in \R^{2N}$ as 
\begin{align} \label{equ:e_S}
    \bm x = \sum_{i=1}^N \bm e_i \otimes \bm S_i(\bm x) ,
\end{align}
where $\bm e_i$ is the $i$th vector in the natural basis of $\R^N$. Now we define the constraint set of the UMLS problem (\ref{prob:umls}) based on the 2-selection operator.
\begin{definition} \label{def:C}
The \textbf{unit-modulus set} is defined by
\begin{align} \label{equ:C}
    \C = \{ \bm x \in \R^{2N} : \norm{\bm S_i(\bm x)}^2 = 1 , \forall i=1,\ldots,N \} .
\end{align}
\end{definition}
\noindent Using Definition~\ref{def:C}, one can rewrite the optimization problem (\ref{prob:umls0}) as follows
\begin{align} \label{prob:umls}
    \min_{\bm x \in \C} \frac{1}{2} \norm{\bm A \bm x - \bm b}^2 .
\end{align}
For convenience, we denote the objective $f(\bm x) = \frac{1}{2} \norm{\bm A \bm x - \bm b}^2$.

\subsection{Projected Gradient Descent for UMLS}

To define the projection onto the unit-modulus set $\C$, let us introduce the distance function from a point $\bm x \in \R^{2N}$ to $\C$ as
\begin{align}
    d(\bm x, \C) = \inf_{\bm y \in \C} \{ \norm{\bm y - \bm x} \} .
\end{align}
The set of all projections of $\bm x$ onto $\C$ is then given by
\begin{align} \label{equ:PC}
    \Pi_\C (\bm x) = \{ \bm y \in \C \mid \norm{\bm y - \bm x} = d(\bm x, \C) \} .
\end{align}
It is well-known \cite{vasilyev2013depth} that if $\C$ is closed, then for any $\bm x \in \R^n$, $\Pi_\C (\bm x)$ is non-empty. Additionally, since the unit-modulus set $\C$ is non-convex, $\Pi_\C (\bm x)$ can have more than one element.
An orthogonal projection onto $\C$ is defined as $\P_\C: \R^{2N} \to \C$ such that $\P_\C (\bm x)$ is chosen as an element of $\Pi_\C (\bm x)$ based on a prescribed scheme (e.g., based on lexicographic order).  
In particular, we define the orthogonal projection $\P_\C(\bm x)$ as projecting each coordinate pair of $\bm x \in \R^{2N}$ onto the unit 1-sphere
\begin{align} \label{equ:Pc}
    \bm S_i \bigl(\P_\C (\bm x)\bigr) = \begin{cases}
        \frac{\bm S_i (\bm x)}{\norm{\bm S_i (\bm x)}} &\text{ if } \bm S_i (\bm x) \neq \bm 0_2 , \\
        [1,0]^\topnew \triangleq \bm s &\text{ if } \bm S_i (\bm x) = \bm 0_2 ,
    \end{cases}
\end{align}
for $i=1,\ldots,N$, where $\bm S_i(\cdot)$ is given in Definition~\ref{def:Si}.
It is noted that when $\bm S_i (\bm x) = \bm 0_2$, the set of projections of $\bm 0_2$ onto the unit 1-sphere is non-singleton, i.e., the entire unit 1-sphere. In such case, we choose a certain element $\bm s$ in this set (e.g., $[1,0]^\topnew$) as the value of $\bm S_i \bigl(\P_\C (\bm x)\bigr)$.
We emphasize that this choice of the projection does not affect our subsequent analysis of local convergence.

Starting from some initial point $\bm x^{(0)}$, the PGD algorithm for solving (\ref{prob:umls}) performs the following iterative update (see Algorithm~\ref{algo:PGD}):
\begin{align} \label{equ:pgd}
    \bm x^{(k+1)} = \P_\C \bigl( \bm x^{(k)} - \eta \bm A^\topnew (\bm A \bm x^{(k)} - \bm b) \bigr) ,
\end{align}
where $\eta>0$ is a fixed step size. In the literature, PGD is also known as the gradient projection (GP) algorithm (e.g., \cite{tranter2017fast}).

\begin{algorithm}[t]
\caption{Projected Gradient Descent (PGD)}
\label{algo:PGD}
\begin{algorithmic}[1]
\Require{$\bm x^{(0)} \in \R^{2N}$}
\Ensure{$\{\bm x^{(k)}\}_{k=0}$}
\For{$k=0,1,\ldots$}
\State $\bm x^{(k+1)} = \P_\C \bigl( \bm x^{(k)} - \eta \bm A^\topnew (\bm A \bm x^{(k)} - \bm b) \bigr)$
\State
\Comment{where $\P_\C$ is defined in (\ref{equ:Pc})}
\EndFor
\end{algorithmic}
\end{algorithm}

\section{Preliminaries}
\label{sec:prel}

This section presents a brief overview of existing results on convergence analysis of PGD for UMLS and the related problem of least squares with unit-norm constraint.

\subsection{Existing Convergence Results on PGD for UMLS}

The sublinear convergence of PGD to a set of stationary points of UMLS was studied in \cite{tranter2017fast}. First, Tranter~\textit{et.~al.} showed that any limiting point $\bm x^*$ of the sequence $\{\bm x^{(k)}\}_{k=0}^\infty$ generated by Algorithm~\ref{algo:PGD} is also a stationary point of (\ref{prob:umls}).
Second, they proved that for PGD with a fixed step size $0<\eta<1/\norm{\bm A}_2^2$, the convergence of $\{\bm x^{(k)}\}_{k=0}^\infty$ to a set of stationary points of (\ref{prob:umls}) is sublinear. In particular, the authors provided a sublinear bound on the distance between two consecutive iterates as follows\footnote{We note that in \cite{tranter2017fast}, the authors actually derived the convergence bound on a surrogate function $Q(\cdot)$ that quantifies the stationarity condition of (\ref{prob:umls}). From Eqn.~(23b) in \cite{tranter2017fast}, we have the value of $Q(\cdot)$ at iteration $k$ equals to $\frac{1}{\eta^2} \norm{\bm x^{(k+1)} - \bm x^{(k)}}^2$. In the literature, such convergence metric is related to the generalized gradient norm, (e.g., \cite{beck2017first}-Section~2.3.2).}
\begin{align} \label{equ:sublinear}
    \min_{0 \leq l \leq k-1} \norm{\bm x^{(l+1)} - \bm x^{(l)}} \leq \sqrt{\frac{2 \eta \bigl(f(\bm x^{(0)}) - f(\bm x^*) \bigr)}{(1-\eta \norm{\bm A}_2^2) k}} .
\end{align}
However, it is noted that the sublinear bound given by (\ref{equ:sublinear}) is based on the worst-case analysis. In practice, we observe the algorithm enjoys fast linear convergence to a local minimum $\bm x^*$ of (\ref{prob:umls}). Figure~\ref{fig:sublinear} illustrates the striking difference between the sublinear bound on $\norm{\bm x^{(k+1)} - \bm x^{(k)}}$ given by the RHS of (\ref{equ:sublinear}) (blue dashed line) and the corresponding linearly converging empirical value obtained by running the PGD algorithm (blue solid line).
The additional bound on $\norm{\bm x^{(k+1)} - \bm x^{(k)}}$ (red dashed line) is derived from the bound on $\norm{\bm x^{(k)} - \bm x^*}$ given by (\ref{equ:exp_decrease}) in the next section and the application of triangle inequality: $\norm{\bm x^{(k+1)} - \bm x^{(k)}} \leq \norm{\bm x^{(k+1)} - \bm x^*} + \norm{\bm x^{(k)} - \bm x^*}$. We observe that the red dashed line and the blue solid line are parallel to each other, while the blue dashed line deviates quickly from the other two lines as $k$ increases.
In the next section, we study this unexplained convergence phenomenon of PGD for UMLS.
We will provide exact formulations of the linear convergence rate and the region of convergence. 
The selection of the fixed step size $0<\eta<1/\norm{\bm A}_2^2$ in \cite{tranter2017fast} is conservative as it may exclude the optimal choice of $\eta$. We will demonstrate in our simulation that larger step sizes enable faster convergence of PGD for UMLS.

\begin{figure}[t]
    \centering
    \includegraphics[width=\columnwidth]{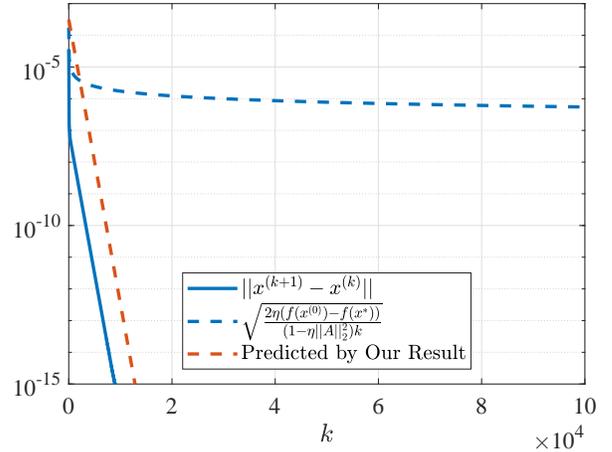}
    \caption{Plot of $\norm{\bm x^{(k+1)}-\bm x^{(k)}}$ (blue solid) generated by PGD for UMLS with a fixed step size $\eta = 0.9/ \norm{\bm A}_2^2$. The blue dashed line represents the sublinear bound given by (\ref{equ:sublinear}). The red dashed line is based on our linear upper bound proposed in this work. Further details of the data generated for this figure are given later in our simulation in Section~\ref{sec:exp}.}
    \label{fig:sublinear}
\end{figure}

\subsection{Least Squares with Unit-Norm Constraint}

\begin{table*}[t]
    \centering
    \begin{tabular}{|l|c|c|}
        \hline 
        & \textbf{Unit-norm constraint} \cite{vu2019convergence} & \textbf{Unit-modulus constraint} (this work) \\ \hline 
        Problem formulation & $\min_{\bm x \in \R^N} \frac{1}{2} \norm{\bm A \bm x - \bm b}^2$ s.t. $\norm{\bm x}=1$ & $\min_{\bm x \in \R^{2N}} \frac{1}{2} \norm{\bm A \bm x - \bm b}^2$ s.t. $\norm{\bm S_i (\bm x)}=1 , \forall i=1,\ldots,N$ \\ \hline
        First-order necessary condition & $\exists \gamma \in \R: \bm A^\topnew (\bm A^\topnew \bm x^* - \bm b) = \gamma \bm x^*$ & $\exists \bm \gamma \in \R^N: \bm A^\topnew (\bm A \bm x^* - \bm b) = (\diag(\bm \gamma) \otimes \bm I_2) \bm x^*$ \\ \hline
        Reduced Riemannian Hessian & $\bm H = \bm Z^\topnew \bm A^\topnew \bm A \bm Z - \gamma \bm I_N$ & $\bm H = \bm Z^\topnew \bm A^\topnew \bm A \bm Z - \diag(\bm \gamma)$ \\ \hline
        Second-order \textit{necessary} condition & $\bm H \succeq \bm 0_N$ & $\bm H \succeq \bm 0_N$ \\ \hline
        Second-order \textit{sufficient} condition & $\bm H \succ \bm 0_N$ & $\bm H \succ \bm 0_N$ \\ \hline
        Fixed-point condition on step size & $1 - \eta \gamma > 0$ & $\bm I_N - \eta \diag(\bm \gamma) \succ \bm 0_N$ \\ \hline
        Convergence condition on the step size & $\eta \bigl(\lambda_1(\bm H) + 2 \gamma \bigr) < 2$ & $\eta \bigl(\lambda_1(\bm H) + 2 \max_i \gamma_i \bigr) < 2$ \protect\footnotemark \\ \hline
        Linear convergence rate & $\rho \bigl( \bm I_N - \eta (1-\eta \gamma)^{-1} \bm H \bigr) $ & $\rho \bigl( \bm I_N - \eta (\bm I_N - \eta \diag(\bm \gamma))^{-1} \bm H \bigr)$ \\ \hline
    \end{tabular}
    \caption{Comparison between the existing convergence analysis of PGD for least squares with unit-norm constraint \cite{vu2019convergence} and the novel convergence analysis of PGD for unit-modulus constraint proposed in this paper. In each case, $\bm x^*$ is a stationary point and $\bm Z$ is a basis matrix for the null space of the Jacobian of all constraints at $\bm x^*$.}
    \label{tbl:compare}
\end{table*}
\footnotetext{This is a more intuitive but not the most general constraint on the step size. The original version of this condition on the step size is given in Theorem~\ref{theo:rate}.}

A closely-related problem to UMLS is the unit-norm least squares (UNLS)
\begin{align*}
    \min_{\bm x \in \R^{N}} \quad &\frac{1}{2} \norm{\bm A \bm x - \bm b}^2 \\
    \text{ s.t. } \quad & \norm{\bm x}^2=1 , \numberthis \label{prob:sphere}
\end{align*}
where $\bm A \in \R^{M \times N}$ and $\bm b \in \R^N$. While UMLS requires each of the $N$ coordinate pairs of the solution lies on the unit $1$-sphere, UNLS requires the solution itself lies on the $N-1$-sphere. Unlike the case of unit-modulus constraint, minimizing a quadratic form over the unit sphere is not NP-hard and is solvable as an eigenvalue problem \cite{sorensen1997minimization,hager2001minimizing}.
The convergence of PGD for UNLS has recently been studied in \cite{beck2018globally,vu2019convergence}.
Table~\ref{tbl:compare} summarizes the existing convergence result on UNLS and the new convergence result on UMLS we derive in this paper, highlighting the connection between the two works.

\section{Convergence Analysis}
\label{sec:analysis}

This section presents the convergence analysis of PGD for UMLS. We begin with the properties of the solution of the problem and the PGD algorithm. Next, we present the main result on the convergence of PGD for UMLS. Finally, we provide the detailed proof at the end of the section.

\subsection{Solution Properties}

The Lagrange function corresponding to (\ref{prob:umls}) is given by
\begin{align*}
    L(\bm x, \bm \gamma) = \frac{1}{2} \norm{\bm A \bm x - \bm b}^2 - \frac{1}{2} \sum_{i=1}^N \gamma_i (x_{2i-1}^2 + x_{2i}^2 - 1) ,
\end{align*}
where $\bm \gamma \in \R^N$ is the Lagrange multiplier.
The derivatives of $L$ with respect to $\bm x$ can be computed as
\begin{align} \label{equ:Lagrange}
\begin{cases}
    \nabla_{\bm x} L(\bm x, \bm \gamma) &= \bm A^\topnew (\bm A \bm x - \bm b) - (\diag(\bm \gamma) \otimes \bm I_2) \bm x , \\
    \nabla_{\bm x}^2 L(\bm x, \bm \gamma) &= \bm A^\topnew \bm A - \diag(\bm \gamma) \otimes \bm I_2 .
\end{cases}
\end{align}
It can be shown that any feasible point $\bm x \in \C$ is also a regular point of the constraint set. Specifically, we first represent the constraints as $\bm h : \R^{2N} \to \R^N$ such that $\bm h (\bm x) = \bm 0_N$, where $h_i(\bm x) = \norm{\bm S_i (\bm x)}^2 - 1$ for $i=1,\ldots,N$. Then, the Jacobian of all the constraints at $\bm x$, defined as $J_{ij}=\partial h_i(\bm x) / \partial x_j$, is given by
\begin{align*}
    \bm J (\bm x) = \begin{bmatrix} \bm e_1^\topnew \otimes \bm S_1^\topnew (\bm x) \\ \ldots \\ \bm e_N^\topnew \otimes \bm S_N^\topnew (\bm x) \end{bmatrix} \in \R^{N \times 2N} .
\end{align*}
Since $\bm J (\bm x)$ is full row rank for any $\bm x \in \C$, $\bm x$ is a regular point of the constraint set (see Chapter~11 in \cite{luenberger1984linear}).
The following lemma establishes the first-order necessary conditions for local optima of UMLS problems.
\begin{lemma} \label{lem:KKT}
The first-order necessary conditions for $\bm x^* \in \R^{2N}$ to be a local minimum of (\ref{prob:umls}) are $\bm x^* \in \C$ and there exists a Lagrange multiplier $\bm \gamma \triangleq \bm \gamma(\bm x^*) \in \R^N$ such that
\begin{align} \label{equ:KKT1}
    \bm A^\topnew (\bm A \bm x^* - \bm b) = (\diag(\bm \gamma) \otimes \bm I_2) \bm x^* .
\end{align}
Any point satisfying the foregoing first-order necessary conditions is called a \textbf{stationary} point of (\ref{prob:umls}).
\end{lemma}
\noindent By setting $\nabla_{\bm x} L(\bm x, \bm \gamma)$ in (\ref{equ:Lagrange}) to $\bm 0$, the proof of Lemma~\ref{lem:KKT} follows the same derivation in \cite{luenberger1984linear}-Chapter~11.3. 
Next, we examine the second-order conditions for local optima of problem (\ref{prob:umls}) via the basis of the tangent space to $\C$ at $\bm x^*$. The following lemma provides further insight into these conditions.
\begin{lemma} \label{lem:KKT2}
Let $\bm x^*$ be a stationary point of problem (\ref{prob:umls}) with the corresponding Lagrange multiplier $\bm \gamma$. A basis of the tangent space to $\C$ at $\bm x^*$ is given by the semi-orthogonal matrix $\bm Z \in \R^{2N \times N}$ such that
\begin{align} \label{equ:Z}
    \bm Z = \sum_{i=1}^N \bm e_i \bm e_i^\topnew \otimes \bm v_i ,
\end{align}
where $\bm v_i = [-x^*_{2i},x^*_{2i-1}]^\topnew$.
Denote the \textbf{reduced Riemannian Hessian} associated with $\bm x^*$ by
\begin{align} \label{equ:H}
    \bm H = \bm Z^\topnew \bm A^\topnew \bm A \bm Z - \diag(\bm \gamma) .
\end{align}
The second-order necessary condition for $\bm x^*$ to be a local minimum of (\ref{prob:umls}) is $\bm H \succeq \bm 0_N$. The second-order sufficient condition for $\bm x^*$ to be a \textbf{strict} local minimum of (\ref{prob:umls}) is $\bm H \succ \bm 0_N$.
\end{lemma}
\noindent The proof of Lemma~\ref{lem:KKT2} is given in Appendix~\ref{appdx:KKT2}.

\begin{remark} \label{rmk:H_Riemannian}
The concept of Riemannian Hessian has been well-studied in differential geometry (e.g., \cite{lee2018introduction}). From (\ref{equ:H}), one can see that the first term takes into account the curvature of the objective function restricted to the unit-modulus manifold $\C$. On the other hand, the second term characterizes the curvature of the manifold $\C$. 
While this is an elementary result in differential geometry, we include the proof detail in Appendix~\ref{appdx:H_Riemannian} for self-containedness.
\end{remark}

\subsection{Algorithm Properties}

The PGD algorithm can be viewed as a fixed-point iteration and hence, can be analyzed via the existing tools from fixed-point theory. We first define the convergent point of the PGD update (\ref{equ:pgd}) as follows.
\begin{definition} \label{def:fixed_point}
The point $\bm x \in \C$ is a \textbf{fixed point} of Algorithm~\ref{algo:PGD} with step size $\eta>0$ if it satisfies
\begin{align} \label{equ:fixed_point}
    \bm x = \P_\C \bigl( \bm x - \eta \bm A^\topnew (\bm A \bm x - \bm b) \bigr) .
\end{align}
\end{definition}
\noindent If the constraint set $\C$ is convex, any fixed point of Algorithm~\ref{algo:PGD} is also an optimal solution of the constrained least squares problem \cite{bertsekas1997nonlinear}. Since the unit-modulus constraint set is non-convex, we show that any fixed point of Algorithm~\ref{algo:PGD} is a stationary point of (\ref{prob:umls}) as follows.
\begin{lemma} \label{lem:fixed}
The vector $\bm x^*$ is a fixed point of Algorithm~\ref{algo:PGD} with step size $\eta>0$ if and only if $\bm x^*$ is a stationary point of the non-convex problem (\ref{prob:umls}) and the corresponding Lagrange multiplier $\bm \gamma$ satisfies 
\begin{align} \label{equ:fixed}
    \begin{cases}
        \gamma_i < 1/\eta &\text{ if } \bm S_i(\bm x^*) \neq \bm s \\
        \gamma_i \leq 1/\eta &\text{ if } \bm S_i(\bm x^*) = \bm s
    \end{cases} \quad \forall i=1,\ldots,N ,
\end{align}
where $\bm s$ is defined in (\ref{equ:Pc}).
\end{lemma}
\noindent The proof of this lemma is given in Appendix~\ref{appdx:fixed}. 
Lemma~\ref{lem:fixed} suggests that when $\eta$ is sufficiently small, all stationary points of (\ref{prob:umls}) can be fixed points of Algorithm~\ref{algo:PGD}. As the step size $\eta$ increases, fewer stationary points satisfying (\ref{equ:fixed}) can be fixed points of the algorithm. 
Next, we study the first-order Taylor expansion of the projection $\P_\C$ about a point in $\C$ in the following proposition:
\begin{proposition} \label{prop:dPc}
For any $\bm x \in \C$ and $\bm \delta \in \R^{2N}$, we have
\begin{align} \label{equ:taylor}
    \P_\C (\bm x + \bm \delta) = \bm x + \bm Z \bm Z^\topnew \bm \delta + \bm q(\bm \delta) ,
\end{align}
where $\bm Z = \sum_{i=1}^N \bm e_i \bm e_i^\topnew \otimes \bm v_i$, for $\bm v_i = [-x_{2i},x_{2i-1}]^\topnew$, and $\bm q: \R^{2N} \to \R^{2N}$ satisfies $\norm{\bm q(\bm \delta)} \leq 2 \norm{\bm \delta}^2$.
\end{proposition}
\noindent The proof of this proposition is given in Appendix~\ref{appdx:dPc}.
It is noteworthy from Proposition~\ref{prop:dPc} that the projection $\P_\C$ is differentiable at any $\bm x \in \C$.
Second, the derivative of $\P_\C$, given by $\bm Z \bm Z^\topnew$, coincides with the orthogonal projection onto the tangent space to $\C$ at $\bm x$ \cite{lewis2008alternating}. 
Third, the expansion (\ref{equ:taylor}) is universal, regardless of the magnitude of $\bm \delta$.

\subsection{Main Result}

We are now in position to state our main result on the linear convergence of PGD for UMLS. 
\begin{theorem} \label{theo:rate}
Consider a stationary point $\bm x^* \in \C$ of the UMLS problem (\ref{prob:umls}) with the corresponding Lagrange multiplier $\bm \gamma \triangleq \bm \gamma(\bm x^*) \in \R^N$ defined in (\ref{equ:KKT1}) and the reduced Riemannian Hessian $\bm H \triangleq \bm H(\bm x^*) \in \R^{N \times N}$ defined in (\ref{equ:H}).
Let $\{ \bm x^{(k)} \}_{k=0}^\infty \subset \R^{2N}$ be the sequence generated by Algorithm~\ref{algo:PGD} with a fixed step size $\eta>0$.
Assume that
\begin{enumerate}[leftmargin=30pt]
    \item[(C1)] $\bm H \succ \bm 0_N$ (sufficient condition for $\bm x^*$ being a strict local minimum),
    \item[(C2)] $\eta \gamma_i \neq 1$ for all $i=1,\ldots,N$, and
    \item[(C3)] $\rho(\bm M_\eta)<1$ where
    \begin{align} \label{def:Mn}
        \bm M_\eta = \bm I_N - \eta \Bigl(\bm I_N - \eta \diag(\bm \gamma) \Bigr)^{-1} \bm H .
    \end{align}
\end{enumerate}
Then, there exists a finite constant $c_0(\bm x^*,\eta)$ \footnote{A closed-form expression of $c_0(\bm x^*,\eta)$ is given in Lemma~\ref{lem:rho2}.} such that for any $\bm x^{(0)} \in \C$ satisfying $\norm{\bm x^{(0)} - \bm x^*} < c_0(\bm x^*,\eta)$, the sequence $\{ \norm{\bm x^{(k)} - \bm x^*} \}_{k=0}^\infty$ converges to $0$. 
Furthermore, if $\norm{\bm x^{(0)} - \bm x^*} < \rho(\bm M_\eta) c_0(\bm x^*,\eta)$, it holds for any integer $k \geq 0$ that
\begin{align} \label{equ:exp_decrease}
    \frac{\norm{\bm x^{(k)} - \bm x^*}}{\norm{\bm x^{(0)} - \bm x^*}} < \biggl(1-\frac{\norm{\bm x^{(0)} - \bm x^*}}{\rho(\bm M_\eta) c_0(\bm x^*,\eta)}\biggr)^{-1} \rho^k(\bm M_\eta) .
\end{align}
In (\ref{equ:exp_decrease}), Algorithm~\ref{algo:PGD} with fixed step size $\eta$ is said to converge \textbf{linearly} to $\bm x^*$ with a rate of $\rho(\bm M_\eta)$.
\end{theorem}
\noindent Theorem~\ref{theo:rate} suggests that PGD in Algorithm~\ref{algo:PGD} initialized near a strict local minimum as indicated by (C1) with a proper step size $\eta$ following the requirements in (C2) and (C3) converges linearly to the local minimum. The theorem establishes three key results for the linear convergence of Algorithm~\ref{algo:PGD}: the region of convergence, the rate of convergence, and the bound on the error through iterations.
Notably, while the previous result in \cite{tranter2017fast} proves the sublinear convergence to a set of stationary points of (\ref{prob:umls}), our result in Theorem~\ref{theo:rate} shows the linear convergence to a strict local minimum.
It is worthwhile mentioning that the linear convergence of $\{ \norm{\bm x^{(k)} - \bm x^*} \}_{k=0}^\infty$ given by (\ref{equ:exp_decrease}) matches with the definition of R-linear convergence in \cite{jorge2006numerical}-Appendix~A.\footnote{Compared to Q-linear convergence, R-linear convergence concerns the overall rate of decrease in the error, rather than the decrease over each individual step of the algorithm. A more elaborate bound on the convergence of non-linear difference equations of the form (\ref{equ:scalar}) is developed in \cite{vu2021closed}, in terms of the number of iterations to reach certain accuracy. In this work, we use a simpler result in Lemmas~\ref{lem:scalar} and \ref{lem:scalar2} to demonstrate the linear convergence.}


Note that Theorem~\ref{theo:rate} does not explicitly suggest an upper bound on $\eta$ that ensures convergence and it may appear that PGD with arbitrarily large step size $\eta$ still converges. However, to ensure convergence, the implicit condition on $\eta$ in (C3) must hold.
To provide an intuition for the step size requirement in this condition, let us consider a more restrictive condition that suffices (C3):
\begin{lemma} \label{lem:C3}
Let $\eta>0$ be a step size such that
\begin{itemize}[leftmargin=40pt]
    \item[(C3')] $\eta (\lambda_1(\bm H) + 2 \overline{\gamma}) < 2$, where $\overline{\gamma} = \max_i \gamma_i$.
\end{itemize}
Then, Condition (C3) in Theorem~\ref{theo:rate} holds, i.e., $\rho(\bm M_\eta)<1$.
\end{lemma}
\noindent The proof of Lemma~\ref{lem:C3} is given in Appendix~\ref{appdx:C3}.
When $\lambda_1(\bm H) + 2 \overline{\gamma} \leq 0$, any step size $\eta>0$ satisfies (C3') and hence, satisfies (C3).
When $\lambda_1(\bm H) + 2 \overline{\gamma}>0$, (C3') suggests an upper bound on $\eta$ that is sufficient but not necessary for (C3), i.e., $\eta<2/(\lambda_1(\bm H) + 2 \overline{\gamma})$.  
As can be seen from Table~\ref{tbl:compare}, Condition (C3') is similar to the convergence condition in the case of unit-norm constraint. 

In Theorem~\ref{theo:rate}, Condition (C3) suggest a non-linear relationship between the convergence rate $\rho(\bm M_\eta)$ and the step size $\eta$. In principle, one can find the optimal step size for local linear convergence by solving the 1-D optimization
\begin{align*}
    \eta^* &= \argmin_{\eta>0} \rho \bigl(\bm M_\eta (\bm x^*)\bigr) \\
    &= \argmin_{\eta>0} \rho \Bigl(\bm I_N - \eta \bigl(\bm I_N - \eta \diag(\bm \gamma(\bm x^*)) \bigr)^{-1} \bm H(\bm x^*) \Bigr) . \numberthis \label{equ:eta_opt}
\end{align*}
In the last equation, we spell out the dependence on $\bm x^*$ to emphasize that the prior knowledge of the local minimum is critical for determining the optimal step size.
In Section~\ref{sec:adaptive}, we propose two variants of PGD with adaptive step size schemes that do not require prior knowledge of $\bm M_\eta$ to select the optimal step size. The proposed algorithms enjoy the fast convergence of PGD with a fixed optimal step size while remaining the same computational complexity per iteration.

\subsection{Proof of Theorem~\ref{theo:rate}}

This subsection presents a proof of Theorem~\ref{theo:rate}, arranging the key ideas into lemmas and deferring their proofs to the appendix. Let us begin with the claim that the strict local minimum $\bm x^*$ in Theorem~\ref{theo:rate} is also a fixed point of PGD with the appropriate choice of the step size $\eta$.
\begin{lemma} \label{lem:PD_fixed}
Consider the same setting as Theorem~\ref{theo:rate}. Assume that Conditions (C1)-(C3) in Theorem~\ref{theo:rate} hold.
Then, $\bm x^*$ is a fixed point of Algorithm~\ref{algo:PGD} with the given step size $\eta$ and its corresponding Lagrange multiplier $\bm \gamma$ satisfies $\gamma_i<1/\eta$, for all $i=1,\ldots,N$.
\end{lemma}
\noindent The proof of Lemma~\ref{lem:PD_fixed} is given in Appendix~\ref{appdx:PD_fixed}.
Next, we establish a recursion on the error vector, based on the first-order approximation of the projection in Proposition~\ref{prop:dPc}.
\begin{lemma} \label{lem:delta1}
Consider the same setting as Theorem~\ref{theo:rate}. Assume that Conditions (C1)-(C3) in Theorem~\ref{theo:rate} hold.
Let $\bm D_\eta=(\bm I_N - \eta \diag(\bm \gamma))^{-1}$ and $\bm \delta^{(k)} = \bm x^{(k)} - \bm x^*$ be the error vector at the $k$th iteration of Algorithm~\ref{algo:PGD}.
Then, for any integer $k\geq 0$, we have
\begin{align*}
    \bm \delta^{(k+1)} &= \bm Z\bm Z^\topnew (\bm D_\eta \otimes \bm I_2) (\bm I_{2N} - \eta \bm A^\topnew \bm A) \bm \delta^{(k)} \\
    &+ \bm q \bigl( (\bm D_\eta \otimes \bm I_2) (\bm I_{2N} - \eta \bm A^\topnew \bm A) \bm \delta^{(k)} \bigr) , \numberthis \label{equ:delta1}
\end{align*}
where $\bm Z$ at $\bm x^*$ and $\bm q$ are defined in Proposition~\ref{prop:dPc}.
\end{lemma}
\noindent The proof of Lemma~\ref{lem:delta1} is given in Appendix~\ref{appdx:delta1}.
Equation (\ref{equ:delta1}) can be viewed as an approximately linear dynamic on the error $\bm \delta^{(k)}$. As the error becomes sufficiently small, the residual term $\bm q ( (\bm D_\eta \otimes \bm I_2) (\bm I_{2N} - \eta \bm A^\topnew \bm A) \bm \delta^{(k)} )$ is negligible while the linear term $\bm Z\bm Z^\topnew (\bm D_\eta \otimes \bm I_2) (\bm I_{2N} - \eta \bm A^\topnew \bm A) \bm \delta^{(k)}$ dominates.
It has been well-studied in the literature \cite{bellman2008stability,polyak1964some,vu2019convergence,vu2021exact,vu2021closed} that the linear convergence rate of (\ref{equ:delta1}) is the spectral radius of the linear operator $\bm Z\bm Z^\topnew (\bm D_\eta \otimes \bm I_2) (\bm I_{2N} - \eta \bm A^\topnew \bm A)$. However, following the argument about the structural constraint on the error vector in \cite{vu2021exact}, we emphasize the fact $\bm \delta^{(k)} = \P_\C (\bm x^* + \bm \delta^{(k)}) - \P_\C(\bm x^*)$ is the difference between two points on the unit-modulus manifold and show that the error vector is dominated by the component on the tangent space to $\C$ at $\bm x^*$.
\begin{lemma} \label{lem:delta2}
Consider the same setting as Theorem~\ref{theo:rate}. At the $k$th iteration of Algorithm~\ref{algo:PGD}, we have
\begin{align} \label{equ:delta_k}
    \bm \delta^{(k)} = \bm Z\bm Z^\topnew \bm \delta^{(k)} + \bm q (\bm \delta^{(k)}) ,
\end{align}
where $\bm Z$ at $\bm x^*$ and $\bm q$ are defined in Proposition~\ref{prop:dPc}.
\end{lemma}
\noindent The proof of Lemma~\ref{lem:delta2} is given in Appendix~\ref{appdx:delta2}.
Next, combining Lemmas~\ref{lem:delta1} and \ref{lem:delta2}, we obtain a recursion on the error vector that implicitly enforces it to lie on the tangent space to $\C$ at $\bm x^*$ as follows.
\begin{lemma} \label{lem:delta3}
Consider the same setting as Theorem~\ref{theo:rate}. Assume that Conditions (C1)-(C3) in Theorem~\ref{theo:rate} hold.
Then by Lemmas~\ref{lem:delta1} and \ref{lem:delta2}, the error vector at the $k$th iteration of Algorithm~\ref{algo:PGD} satisfies
\begin{align} \label{equ:H_delta}
    {\bm \delta}^{(k+1)} = \bm Z \bm M_\eta \bm Z^\topnew {\bm \delta}^{(k)} + \hat{\bm q}({\bm \delta}^{(k)}) ,
\end{align}
where $\bm Z$ at $\bm x^*$ is defined in Proposition~\ref{prop:dPc}, $\hat{\bm q}: \R^{2N} \to \R^{2N}$ satisfies $\norm{\hat{\bm q} ({\bm \delta})} \leq 2 c_\eta (c_\eta + 1) \norm{{\bm \delta}}^2$, and $c_\eta=\norm{((\bm I_N - \eta \diag(\bm \gamma))^{-1} \otimes \bm I_2) (\bm I_{2N} - \eta \bm A^\topnew \bm A)}_2$.
\end{lemma}
\noindent The proof of Lemma~\ref{lem:delta3} is given in Appendix~\ref{appdx:delta3}.
Finally, we show the convergence of $\{\bm \delta^{(k)}\}_{k=0}^\infty$ by recognizing that \textit{(i)} the spectral radius of $\bm Z \bm M_\eta \bm Z^\topnew$ is the same as that of $\bm M_\eta$ and \textit{(ii)} the recursion (\ref{equ:H_delta}) is an approximately linear difference equation that is convergent for $\bm \delta^{(0)}$ sufficiently close to $\bm 0_{2N}$.
\begin{lemma} \label{lem:rho2}
Consider the same setting as Theorem~\ref{theo:rate}. Assume that Conditions (C1)-(C3) in Theorem~\ref{theo:rate} hold.
Let us define $\overline{\gamma} = \max_i \gamma_i$, $\underline{\gamma} = \min_i \gamma_i$ and
\begin{align} \label{equ:c0}
    c_0(\bm x^*,\eta) &= \frac{1-\rho(\bm M_\eta)}{2c_\eta(c_\eta+1)} \frac{1-\eta \overline{\gamma}}{1-\eta \underline{\gamma}} ,
\end{align}
where $c_\eta$ is defined in Lemma~\ref{lem:delta3}.
If $\norm{\bm \delta^{(0)}} < c_0(\bm x^*,\eta)$, then the sequence $\{ \bm \delta^{(k)} \}_{k=0}^\infty$ converges to $\bm 0_{2N}$.
Furthermore, let $c_1(\bm x^*,\eta) = \rho(\bm M_\eta) c_0(\bm x^*,\eta)$.
Then, for any $\norm{\bm \delta^{(0)}} < c_1(\bm x^*,\eta)$ and integer $k \geq 0$, we have
\begin{align} \label{equ:rho1_norm_delta}
    \norm{\bm \delta^{(k)}} \leq \biggl(1-\frac{\norm{\bm \delta^{(0)}}}{c_1(\bm x^*,\eta)}\biggr)^{-1} \biggl( \frac{1- \eta \underline{\gamma}}{1- \eta \overline{\gamma}} \biggr)^{1/2} \norm{\bm \delta^{(0)}} \rho^k(\bm M_\eta) .
\end{align}
\end{lemma}
\noindent The proof of Lemma~\ref{lem:rho2} is given in Appendix~\ref{appdx:rho2}.
With this lemma, we complete our proof of Theorem~\ref{theo:rate}.






\section{Implementation Aspects}
\label{sec:adaptive}

This subsection describes two practical variants of PGD with adaptive step size that can be used when no prior knowledge of the solution is available: PGD with backtracking line search (Algorithm~\ref{algo:bt}) and Nesterov's accelerated PGD with adaptive restart (Algorithm~\ref{algo:arnag}). 

\subsection{Backtracking PGD (Bt-PGD)}

In backtracking PGD, the step size is chosen to approximately minimize the objective function $f(\bm x) = \frac{1}{2} \norm{\bm A \bm x - \bm b}^2$ along the ray $\{ \bm x -\eta \tilde{\bm g}_{\eta} \mid \eta>0 \}$, where
\begin{align*}
    \tilde{\bm g}_{\eta} = \frac{1}{\eta} \Bigl( \bm x - \P_\C\bigl(\bm x - \eta \bm A^\topnew (\bm A \bm x - \bm b) \bigr) \Bigr)
\end{align*}
is the generalized gradient. To guarantee certain decrease in the objective function, we use the following backtracking condition \cite{beck2017first} 
\begin{align} \label{equ:f1}
    f(\bm x - \eta \tilde{\bm g}_{\eta}) \leq f(\bm x) - \eta \tilde{\bm g}_{\eta}^\topnew \nabla f(\bm x) + \frac{\eta}{2} \norm{\tilde{\bm g}_{\eta}}^2 .
\end{align}
Since $f(\cdot)$ is a quadratic, it can be expanded as
\begin{align} \label{equ:f2}
    f(\bm x - \eta \tilde{\bm g}_{\eta}) = f(\bm x) - \eta \tilde{\bm g}_{\eta}^\topnew \nabla f(\bm x) + \eta^2 \tilde{\bm g}_{\eta_k}^\topnew \nabla^2 f \tilde{\bm g}_{\eta_k} .
\end{align}
Substituting (\ref{equ:f2}) back into the LHS of (\ref{equ:f1}) and using the fact that $\nabla^2 f = \bm A^\topnew \bm A$, we obtain the simplified backtracking condition $\tilde{\bm g}_{\eta_k}^\topnew \bm A^\topnew \bm A \tilde{\bm g}_{\eta_k} \leq \frac{1}{\eta_k} \norm{\tilde{\bm g}_{\eta_k}}^2$ as in Algorithm~\ref{algo:bt}-Line~8.
It is worthwhile to note that a factor of $1/\alpha$ is applied to increase the step size at the end of each iteration to encourage the algorithm to explore larger step sizes with faster convergence. We emphasize that this strategy is different from the well-known backtracking line search method in the literature (e.g., \cite{boyd2004convex}), in which the step size $\eta$ is reset to $1$ before the backtracking line search is performed. As a result, the constant $\alpha$ in Algorithm~\ref{algo:bt} should not be interpreted as the fraction of the decrease in the objective function as in \cite{boyd2004convex}-Algorithm~9.2.

\begin{algorithm}[t]
\caption{Backtracking PGD (Bt-PGD)}
\label{algo:bt}
\begin{algorithmic}[1]
\Require{$\bm x^{(0)} \in \R^{2N}$, $\alpha \in (0,1]$, $\beta \in (0,1)$}
\Ensure{$\{\bm x^{(k)}\}_{k=0}$}
\State $\eta_0=1$
\For{$k=0,1,2,\ldots$}
\State $\bm g_k = \bm A^\topnew (\bm A \bm x^{(k)} - \bm b)$
\State $\eta_k=\eta_k/\beta$
\Repeat
    \State $\eta_k = \beta \eta_k$
    \State $\tilde{\bm g}_{\eta_k} = ( \bm x^{(k)} - \P_\C ( \bm x^{(k)} - \eta_k \bm g_k ) ) / \eta_k$
\Until{$\tilde{\bm g}_{\eta_k}^\topnew \bm A^\topnew \bm A \tilde{\bm g}_{\eta_k} \leq \frac{1}{\eta_k} \norm{\tilde{\bm g}_{\eta_k}}^2$}
\State $\bm x^{(k+1)} = \bm x^{(k)} - \eta_k \tilde{\bm g}_{\eta_k}$
\State $\eta_{k+1}=\eta_k/\alpha$
\EndFor
\end{algorithmic}
\end{algorithm}

\subsection{Adaptive Restart Nesterov's Accelerated PGD (ARNAPGD)}
Next, we present an acceleration technique for PGD, named adaptive restart Nesterov's accelerated projected gradient descent (ARNAPGD).
In unconstrained optimization, it has been well-known that Nesterov's accelerated gradient (NAG) \cite{nesterov2003introductory} can dramatically improve the linear convergence rate of gradient descent (GD) for minimizing a $\mu$-strongly convex, $L$-smooth function. As pointed out in \cite{lessard2016analysis}-Proposition~12, GD with a fixed step size $\alpha=1/L$ has convergence rate $\rho \leq \sqrt{(L-\mu)/(L+\mu)}$, while NAG with fixed parameters $\alpha=1/L$ and $\beta=(\sqrt{L}-\sqrt{\mu})/(\sqrt{L}+\sqrt{\mu})$ has convergence rate $\rho \leq \sqrt{1-\sqrt{\mu/L}}$.
Since NAG requires a specific choice of parameters that depends on $L$ and $\mu$, O'Donoghue and Candes \cite{o2015adaptive} proposed a more practical variant called the Nesterov's accelerated gradient with adaptive restart (ARNAG) that recovers the same rate of convergence with no prior knowledge of function parameters.
In this work, we modify ARNAG with gradient scheme to the context of PGD for constrained optimization.
Specifically, each iteration uses backtracking line search for determining the \textit{projected} gradient step $\eta$ and the \textit{generalized} gradient scheme for determining when to restart the momentum. The advantage of this acceleration is it has the same computational complexity per iteration as PGD and Bt-PGD\footnote{The number of matrix-vector products in ARNAPGD is exactly the same as that in Bt-PGD.} while achieving significantly faster convergence rate.
Further details on ARNAPGD are provided in Algorithm~\ref{algo:arnag}. In the next section, we compare the performance of PGD with a fixed optimal step size, Bt-PGD, and ARNAPGD for UMLS.

\begin{algorithm}[t]
\caption{Adaptive restart Nesterov's accelerated PGD (ARNAPGD) with gradient scheme}
\label{algo:arnag}
\begin{algorithmic}[1]
\Require{$\bm x^{(0)} \in \R^{2N}$, $\alpha \in (0,1]$, $\beta \in (0,1)$}
\Ensure{$\{\bm x^{(k)}\}_{k=0}$}
\State $\eta_0=1$
\State $\theta_0=1$
\State $\bm y^{(0)} = \bm x^{(0)}$
\For{$k=0,1,2,\ldots$}
\State $\bm g_k = \bm A^\topnew (\bm A \bm y^{(k)} - \bm b)$
\State $\eta_k=\eta_k/\beta$
\Repeat
    \State $\eta_k = \beta \eta_k$
    \State $\tilde{\bm g}_{\eta_k} = ( \bm y^{(k)} - \P_\C ( \bm x^{(k)} - \eta_k \bm g_k ) ) / \eta_k$
\Until{$\tilde{\bm g}_{\eta_k}^\topnew \bm A^\topnew \bm A \tilde{\bm g}_{\eta_k} \leq \frac{1}{\eta_k} \norm{\tilde{\bm g}_{\eta_k}}^2$}
\State $\bm x^{(k+1)} = \bm y^{(k)} - \eta_k \tilde{\bm g}_{\eta_k}$
\State $\theta_{k+1} = \frac{2\theta_k}{\theta_k + \sqrt{\theta_k^2+4}}$
\State $\beta_{k+1} = \theta_k (1-\theta_k)/(\theta_k^2+\theta_{k+1})$
\State $\bm y^{(k+1)} = \bm x^{(k+1)} + \beta_{k+1} (\bm x^{(k+1)} - \bm x^{(k)})$
\State $\eta_{k+1}=\eta_k/\alpha$
\If{$\tilde{\bm g}_{\eta_k}^\topnew (\bm x^{(k+1)} - \bm x^{(k)}) > 0$}
    \State $\theta_{k+1}=0$
\EndIf
\EndFor
\end{algorithmic}
\end{algorithm}

\section{Numerical Evaluation}
\label{sec:exp}

This section demonstrates the correctness of our theoretical result on the linear convergence of PGD for UMLS in Theorem~\ref{theo:rate}. We show through numerical simulation that our predicted rate of convergence matches the decrease in the distance to the solution through iterations. Moreover, we illustrate the effectiveness of the two variants of PGD with adaptive step sizes proposed in Section~\ref{sec:adaptive}. Finally, we present a simple 2-D example of the region of convergence to demonstrate our theoretical bound in (\ref{equ:c0}).

\subsection{PGD with a Fixed Step Size}

\begin{figure*}
    \centering
    \begin{subfigure}[b]{0.48\textwidth}
        \centering
        \includegraphics[width=\textwidth]{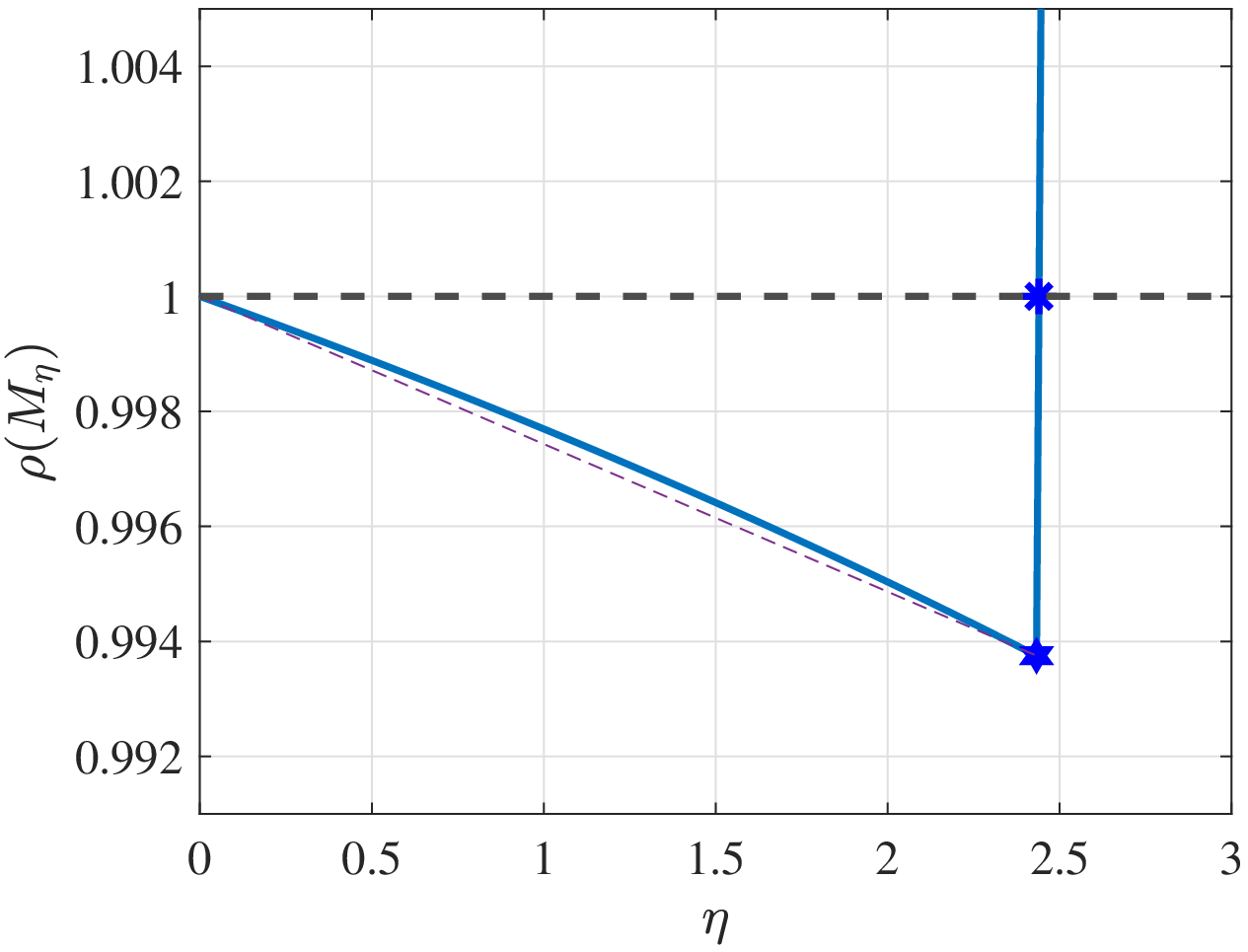}
        \caption{}
    \end{subfigure}
    \quad
    \begin{subfigure}[b]{0.48\textwidth}
        \centering
        \includegraphics[width=\textwidth]{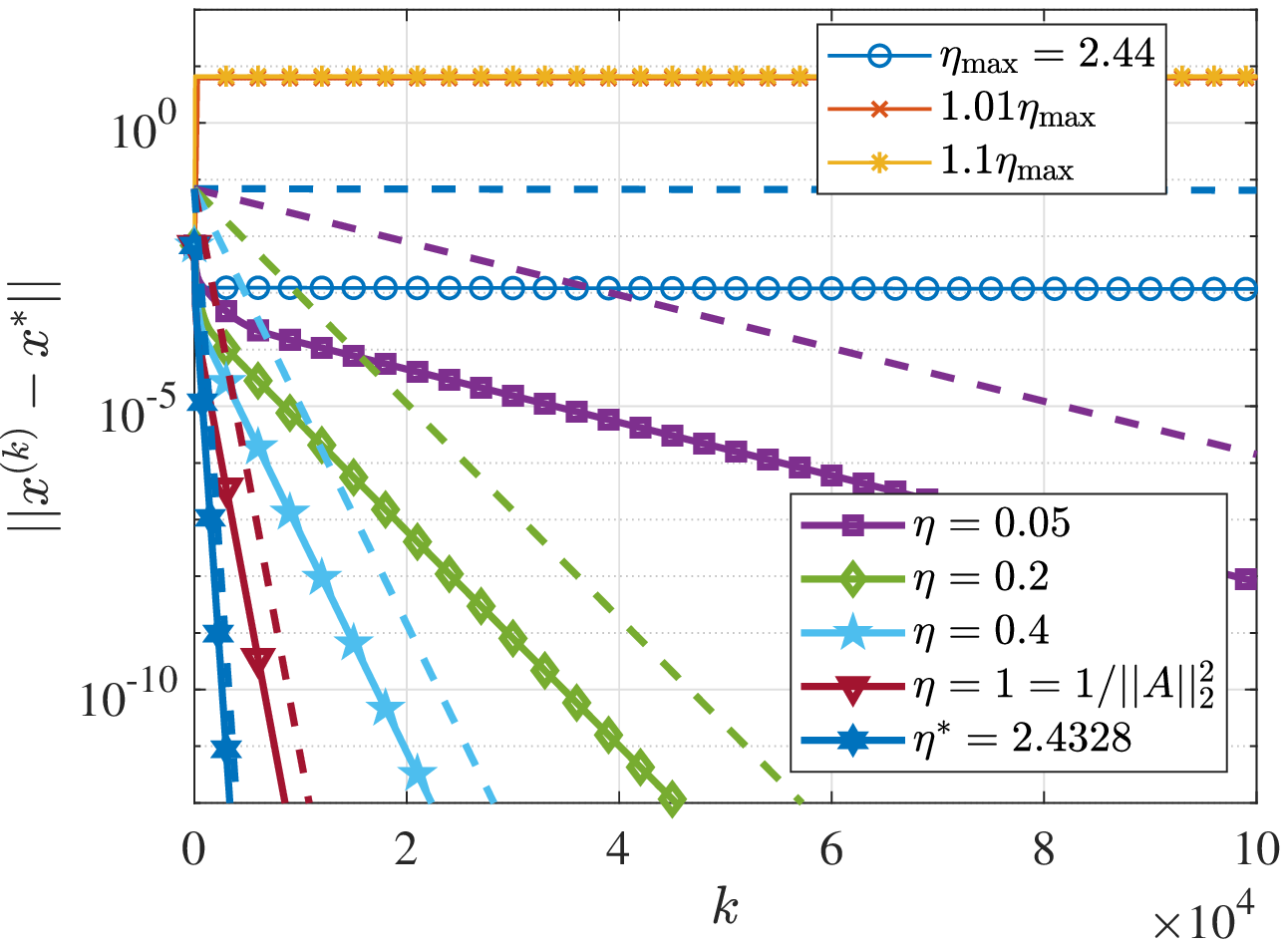}
        \caption{}
    \end{subfigure}
    \caption{Convergence of PGD with a fixed step size for UMLS. (a) Plot of the convergence rate $\rho(\bm M_\eta)$ as a function of the step size $\eta$. The black dashed line is the line $\eta=1$, emphasizing that the local convergence is guaranteed when $\rho(\bm M_\eta)<1$. The blue star represents the maximum step size $\eta_{\max}$ such that $\rho(\bm M_{\eta_{\max}})=1$, while the blue hexagram represents the optimal step size is $\eta^* = \argmin_{\eta>0} \rho(\bm M_\eta)$. The purple dashed line connecting the blue hexagram and the point (0,1) is included to demonstrate the non-linearity of $\rho(\bm M_\eta)$ for $\eta \in (0,\eta^*)$. (b) Plot of the distance between the current update and the local minimum as a function of the number of iterations for various fixed step sizes. Dashed lines represent the corresponding upper bounds with exponential decay, i.e., $\rho^k(\bm M_\eta)$ up to a constant.} 
    \label{fig:umls}
\end{figure*}

\noindent \textbf{Data generation}.
In the following, we create an UMLS setting in which $\bm x^* \in \C$ satisfies
\begin{align*}
    \begin{cases}
        \bm A^\topnew (\bm A \bm x^* - \bm b) = (\diag(\bm \gamma) \otimes \bm I_2) \bm x^* \\
        \bm H = \bm Z^\topnew \bm A^\topnew \bm A \bm Z - \diag(\bm \gamma) \succ \bm 0_N 
    \end{cases}
\end{align*}
as follows.
First, we generate two matrices $\Re$ and $\Im$ of size $M \times N$, where $M=50$ and $N=40$, with i.i.d normally distributed ($\N(0,1)$) entries. The matrix $\bm A$ is computed from $\Re$ and $\Im$ using (\ref{equ:A_ij}).
Second, we generate a random vector $\bm v \in \R^N$ with $i.i.d$ normally distributed entries following $\N(0,0.1^2)$ and a random vector $\bm t \in \{-1,1\}^N$ with uniformly distributed entries.
Then, we obtain $\bm x^*$ and $\bm \gamma$ by setting
\begin{align*}
\begin{cases}
    \gamma_i = t_i \norm{\bm S_i(\bm A^\topnew \bm v)} \\
    \bm S_i(\bm x^*) = \bm S_i(\bm A^\topnew \bm v) / \gamma_i
\end{cases}
     \quad \text{ for } i=1,\ldots,N .
\end{align*}
Next, the matrices $\bm Z$ and $\bm H$ are obtained by (\ref{equ:Z}) and (\ref{equ:H}), respectively.
If $\bm H$ is not PD, we re-run the foregoing generation process multiple times until $\bm H \succ \bm 0_N$. This guarantees Condition (C1) in Theorem~\ref{theo:rate} is satisfied. 
Finally, we compute $\bm b = \bm A \bm x^* - \bm v$ and initialize $\bm x^{(0)}$ near $\bm x^*$ by adding a random noise with $i.i.d$ normally distributed entries following $\N(0,0.001^2)$ to $\bm x^*$.

\noindent \textbf{Results.} Figure~\ref{fig:umls}(a) demonstrates the convergence rate $\rho(\bm M_\eta)$ (blue solid line) as a function of the step size $\eta$. Recall that $\bm M_\eta = \bm I_N - \eta (\bm I_N - \eta \diag(\bm \gamma) )^{-1} \bm H$ and hence, $\rho(\bm M_\eta)$ is a non-linear function of $\eta$. It can be seen from the plot that $\rho(\bm M_\eta)$ approaches $1$ (slow convergence) when $\eta$ approaches either $0$ or $\eta_{\max}=2.44$. The optimal step size that yields the fastest convergence for PGD with a fixed step size is $\eta^*=\argmin_{\eta>0} \rho(\bm M_\eta)=2.4328$. 
Figure~\ref{fig:umls}(b) shows the convergence of PGD with various fixed step sizes.
We observe that for $\eta>\eta_{\max}$ (the overlapping red and yellow solid lines), the algorithm diverges from the designed strict local minimum $\bm x^*$.
For step sizes less than $\eta_{\max}$, our theoretical rate (dashed lines) matches well with the empirical rate (solid lines).
Moreover, PGD with the optimal step size $\eta^*$ converges roughly twice as fast as one with the step size $\eta=1/\norm{\bm A}_2^2$ proposed in \cite{tranter2017fast}, suggesting that the latter choice, while being commonly used in the literature, is conservative.

\subsection{Adaptive Schemes for Step Size}

\begin{figure*}
    \centering
    \begin{subfigure}[b]{0.48\textwidth}
        \centering
        \includegraphics[width=\textwidth]{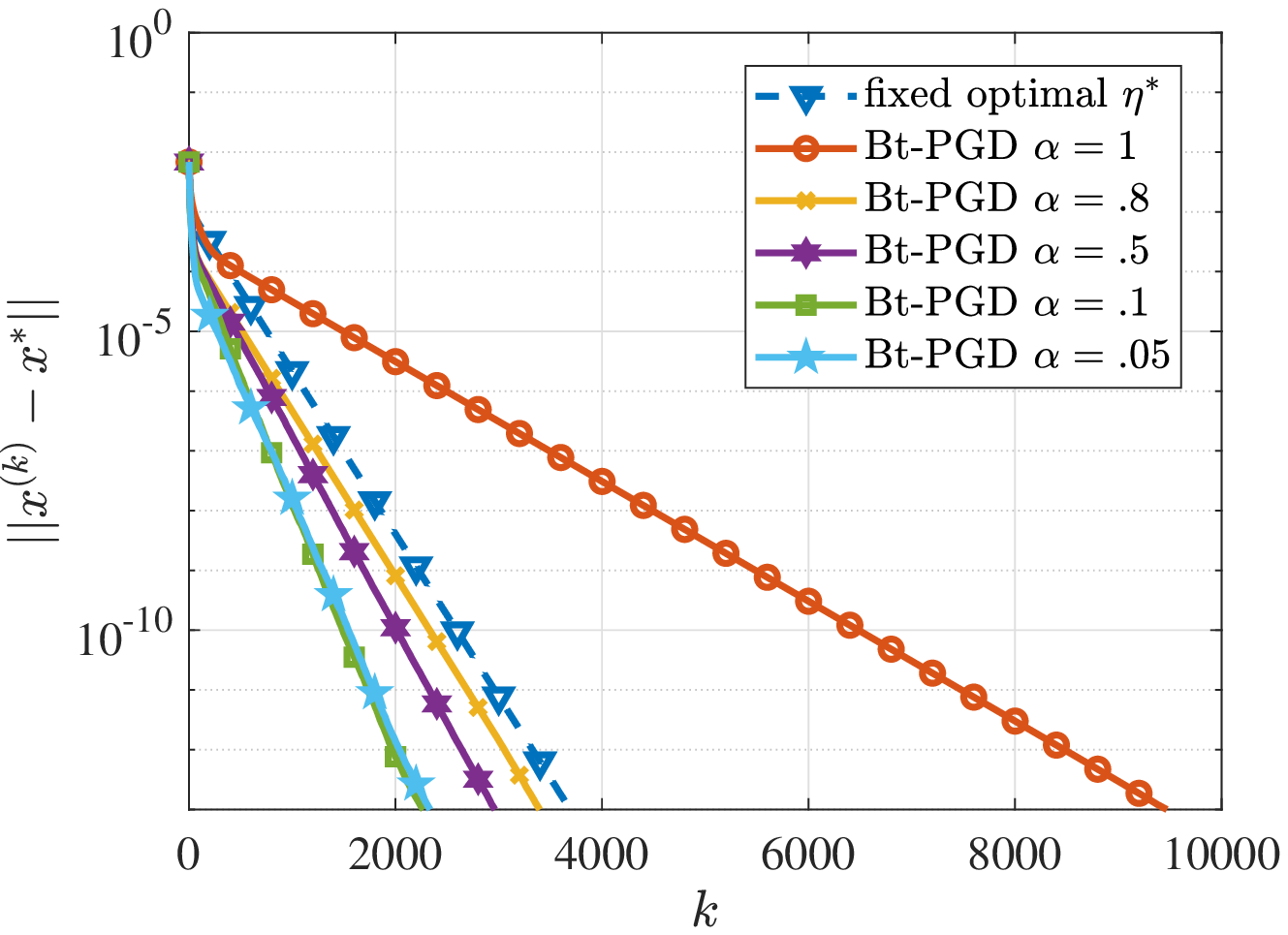}
        \caption{}
    \end{subfigure}
    \quad
    \begin{subfigure}[b]{0.48\textwidth}
        \centering
        \includegraphics[width=\textwidth]{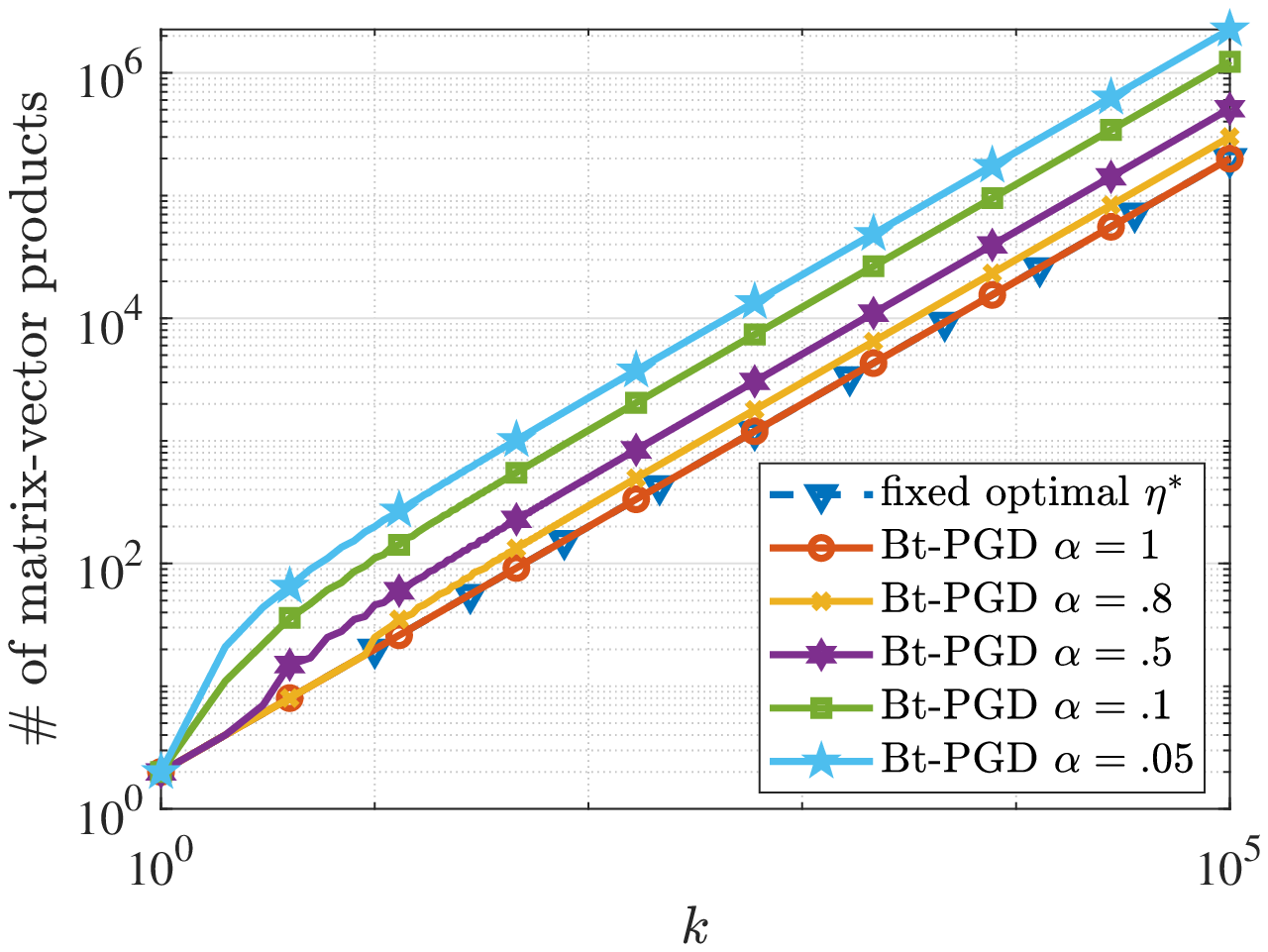}
        \caption{}
    \end{subfigure}
    \quad
    \begin{subfigure}[b]{0.48\textwidth}
        \centering
        \includegraphics[width=\textwidth]{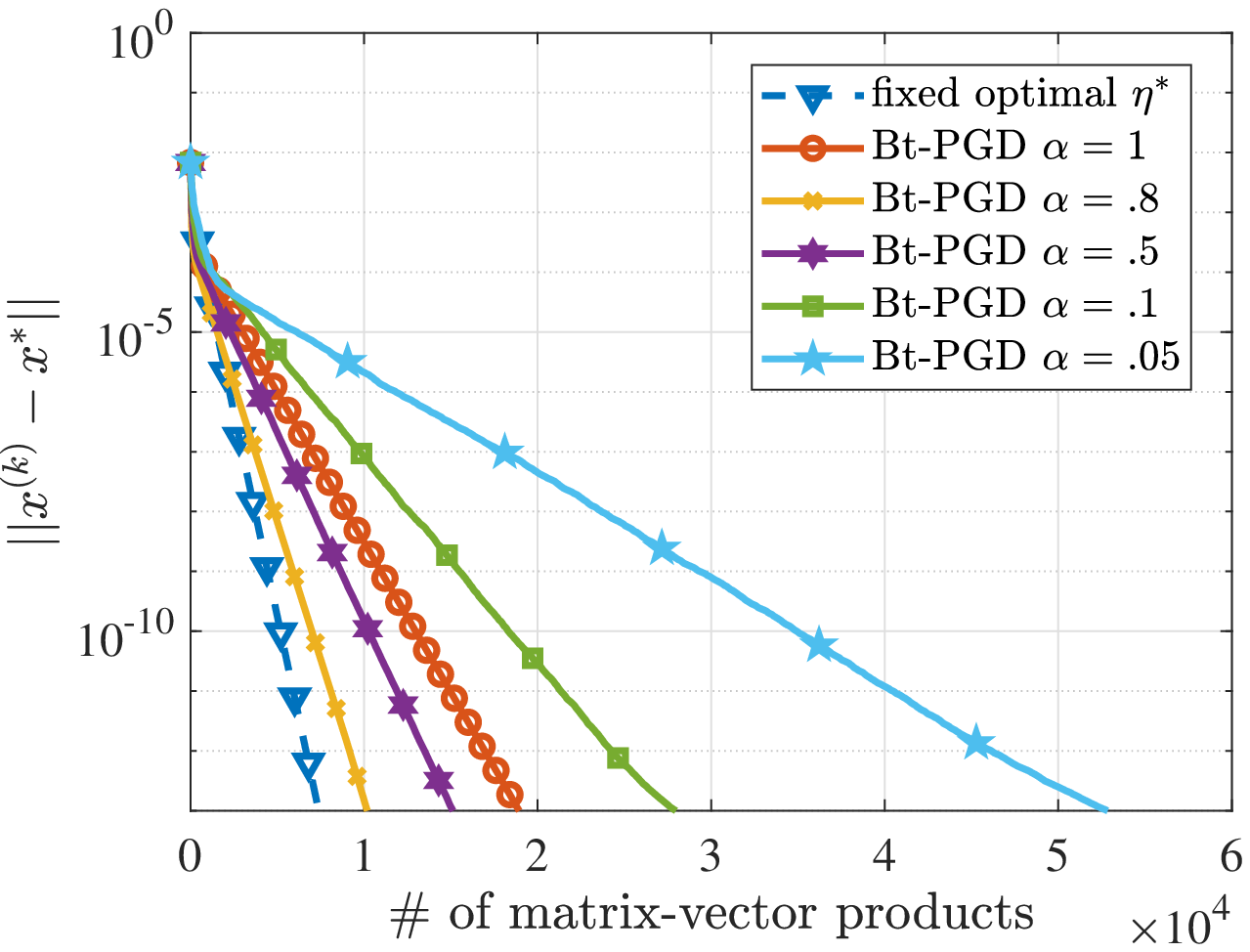}
        \caption{}
    \end{subfigure}
    \quad
    \begin{subfigure}[b]{0.48\textwidth}
        \centering
        \includegraphics[width=\textwidth]{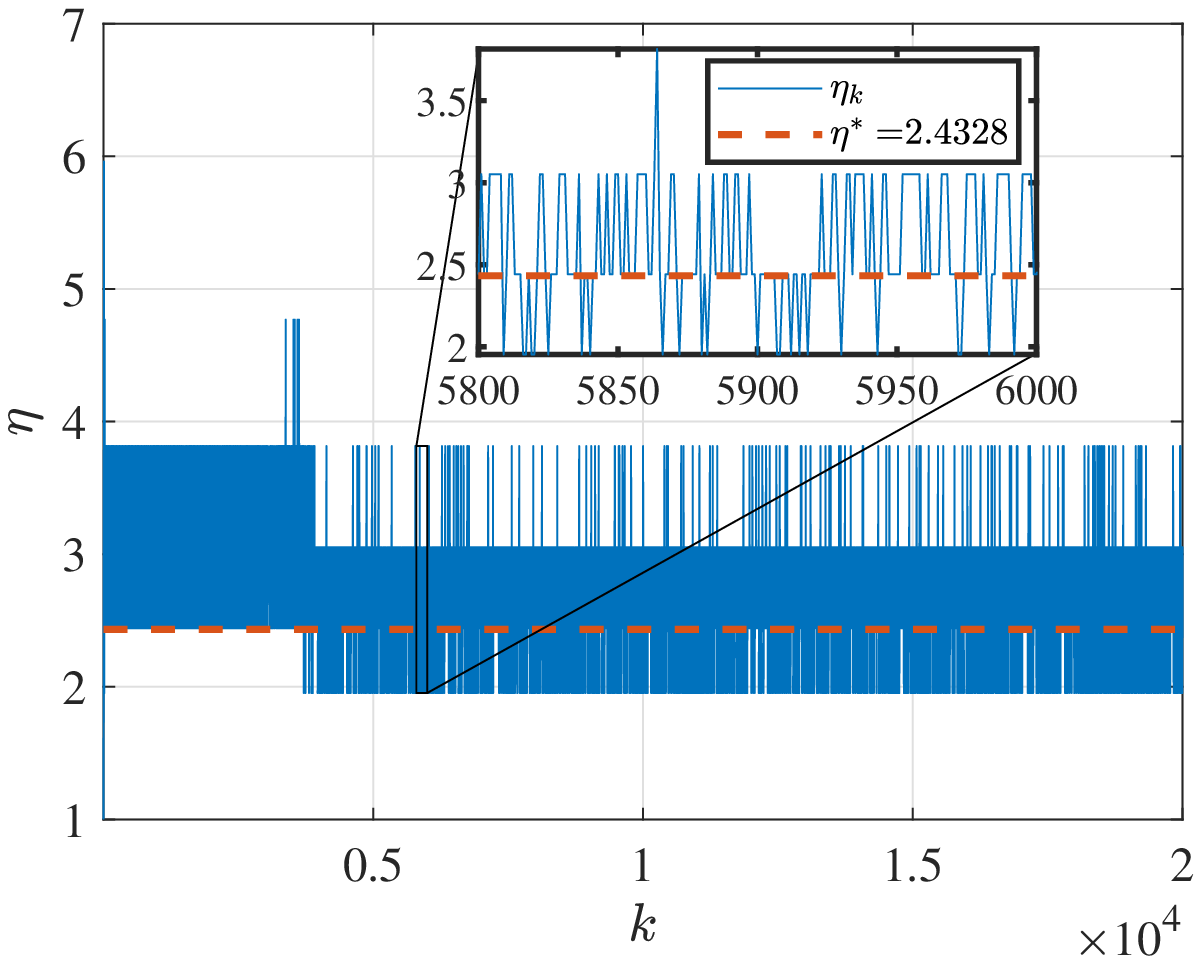}
        \caption{}
    \end{subfigure}
    
    \caption{Convergence of Bt-PGD with various values of $\alpha$ and a fixing value of $\beta=0.8$. (a) Plot of the distance from the current update of Bt-PGD to the local minimum as a function of the number of iterations. A dashed blue line is included as an illustration of the convergence of PGD with the fixed optimal step size $\eta^*$. (b) Plot of the number of matrix-vector products used by Bt-PGD as a function of the number of iterations. (c) Plot of the distance from the current update of Bt-PGD to the local minimum as a function of the number of matrix-vector products. (d) Plot of the backtracking step size $\eta$ as a function of the number of iterations for Bt-PGD with $\alpha=\beta=0.8$. A zoom-in plot is included on top of the original plot for enhanced visualization. After a few thousand iterations, we observe that the adaptive step size $\eta_k$ fluctuates around the optimal step size $\eta^*=2.4328$ (red dashed line).}
    \label{fig:bt}
\end{figure*}

To illustrate the role of $\alpha$ in exploring larger step sizes with faster convergence while balancing the cost of backtracking steps, we plot the error through iterations $\norm{\bm x^{(k)} - \bm x^*}$ against the number of matrix-vector products, which dominates the computational complexity per iteration, in  Fig.~\ref{fig:bt}. 
The data used in this simulation is the same as in the previous section.
While the smaller values of $\alpha$ seems to yields faster convergence (see Fig.~\ref{fig:bt}(a)), they indeed require more backtracking steps at each iteration (see Fig.~\ref{fig:bt}(b)). As a result, the overall computation is higher for smaller values of $\alpha$. It can be seen from Fig.~\ref{fig:bt}(c) that the best choice of $\alpha$ is $\alpha=\beta=0.8$. In addition, we observe that the total cost of Bt-PGD is comparable to that of PGD with the optimal fixed step size. However, Bt-PGD does not use any prior knowledge about the solution $\bm x^*$. 
Fig.~\ref{fig:bt}(d) shows the fluctuation in the step size $\eta$ around the optimal value $\eta^*=2.4328$. It is interesting to note that even though $\eta>\eta_{\max}$ at some iterations, the algorithm is able to converge to the designed local minimum $\bm x^*$.

Figure~\ref{fig:adaptive} compares the performance of four algorithms in solving the foregoing UMLS setting: PGD with a fixed step $\eta=1/\norm{\bm A}_2^2$ (used in \cite{tranter2017fast}), PGD with a fixed optimal step $\eta^*$ (given by (\ref{equ:eta_opt})), Bt-PGD with the optimal choice $\alpha=\beta=0.8$ (Algorithm~\ref{algo:bt}), and ARNAPGD (Algorithm~\ref{algo:arnag}).
The data used in this simulation is the same as in the previous section.
We observe that all three algorithms proposed in this work outperforms PGD with the commonly used step size $\eta=1/\norm{\bm A}_2^2$ (blue solid line).
It is also highlighted that ARNAPGD (purple solid line) obtains significantly faster convergence compared to the other algorithms while remaining similar computational complexity per iteration.
Finally, we note that both Algorithm~\ref{algo:bt} and Algorithm~\ref{algo:arnag} do not come with convergence guarantees in our setting since $\C$ is non-convex. Nonetheless, on the practical side, they do not require prior knowledge of the solution and their effectiveness is depicted clearly through our numerical results. 

\begin{figure}[t]
    \centering
    \includegraphics[width=\columnwidth]{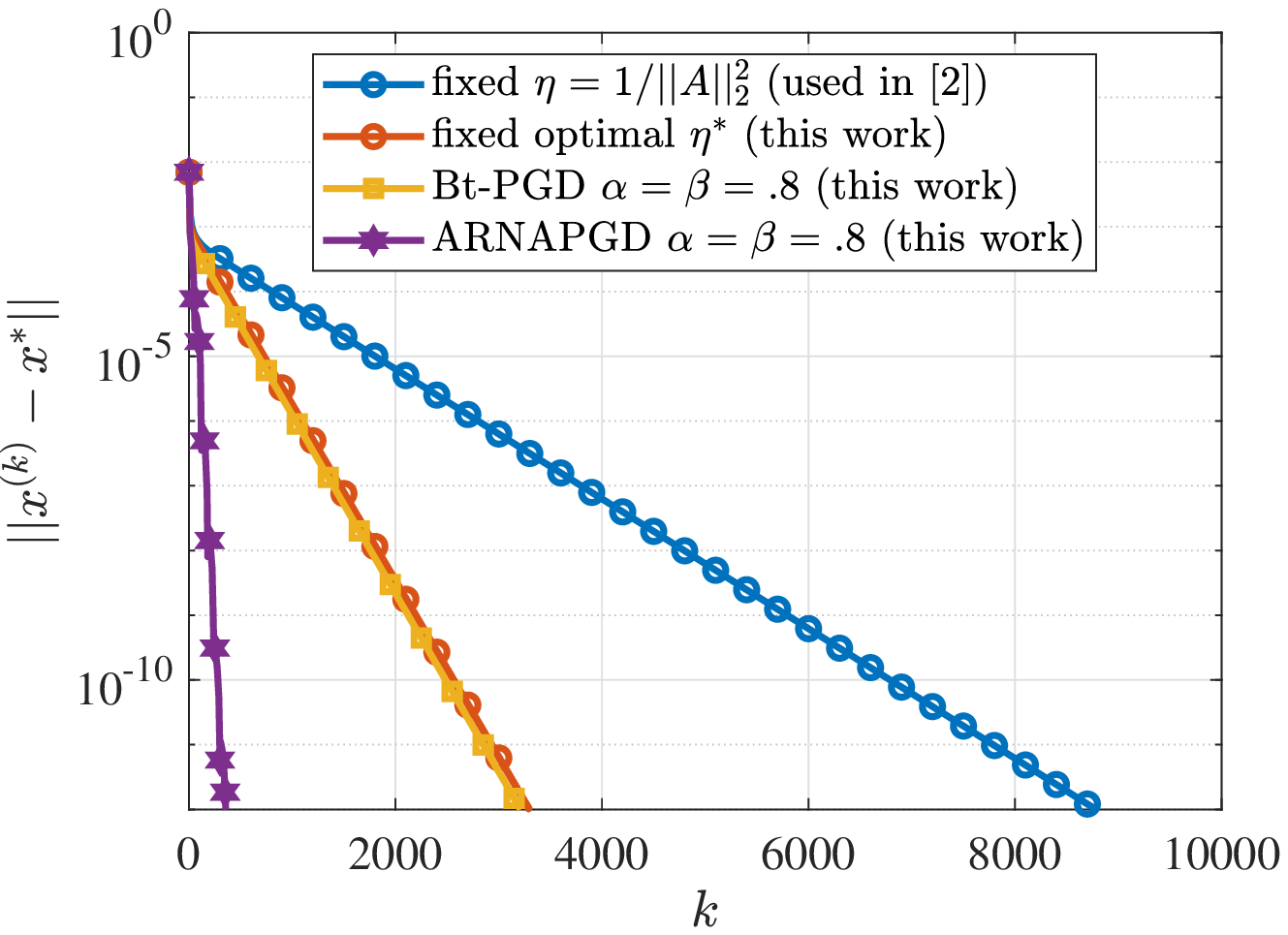}
    \caption{Plot of the distance from the current update to the local minimum $\bm x^*$ as a function of the number of iterations, for four algorithms: PGD with a fixed step $\eta=1/\norm{\bm A}_2^2$ (blue solid line), PGD with a fixed optimal step $\eta^*$ given by (\ref{equ:eta_opt}) (red solid line), Bt-PGD with $\alpha=\beta=0.8$ (yellow solid line), and ARNAPGD (red solid line). All algorithms have the same computational complexity per iteration.}
    \label{fig:adaptive}
\end{figure}

\subsection{Region of Convergence}

\begin{figure}[t]
    \centering
    \includegraphics[width=\columnwidth]{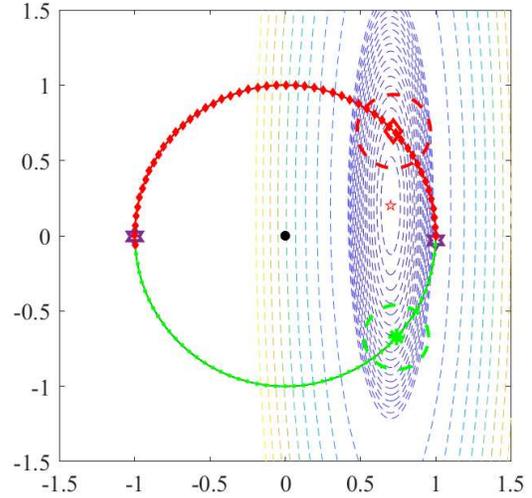}
    \caption{A 2-D illustration of the region of convergence given by the constant $c_0(\bm x^*,\eta)$ in (\ref{equ:c0}). On the circle, the two purple hexagrams denote the local maxima, while the green asterisk and the red diamond denote the local minima of the problem. The red star located inside the circle is the solution to the unconstrained least squares. For a given fixed step size $\eta$, each local minimum is associated with (i) an estimated region of convergence (dashed circle) given by $c_0(\bm x^*,\eta)$ and (ii) an empirical region of convergence (circular arc with matching color) given by running PGD with the fixed step size $\eta$ and initialization at a given point on the circle to verify which local minimum it converges to.}
    \label{fig:example}
\end{figure}

In this subsection, we demonstrate the region of local convergence for PGD in a 2-D setting.
Since $N=1$ in this case, the constraint set $\C$ is indeed a 2-D circle. As can be seen from Fig.~\ref{fig:example}, the least squares objective has an unconstrained global minimum at $\bm x_{unc}^* = [0.7,0.2]^\topnew$, with $\bm A = \diag([5,1])$ and $\bm b = [3.5, 0.2]^\topnew$.
Using Lemma~\ref{lem:KKT}, we can find the four stationary points of the 2-D UMLS problem by solving the following system of non-linear equations
\begin{align*}
\begin{cases}
    x_1^2 + x_2^2 = 1 , \\
    25 x_1 - 17.5 = \gamma x_1 , \\
    x_2 - 0.2 = \gamma x_2 .
\end{cases}
\end{align*}
Moreover, based on the positivity of the reduced Riemannian Hessian $h = 25 x_2^2 + x_1^2 - \gamma$ (which is a scalar in the 2-D setting), one can apply Lemma~\ref{lem:KKT2} to determine the two local maxima (purple hexagrams) and two local minima (green asterisk and red diamond). 
Additionally, for each local minima, the rate of convergence is given by $\rho_\eta = 1 - \eta h / (1-\eta \gamma)$, with the maximum possible step size $\eta_{\max} = 2/(h+2\gamma)$.
In Fig.~\ref{fig:example}, we pick $\eta=0.0755$ and compute the theoretical region of convergence for each local minima using (\ref{equ:c0}). On the other hand, the empirical region of convergence is obtained follows. First, we run PGD with $\eta=0.0755$ and $1000$ different initialization uniformly distributed on the unit circle. Next, we check whether the algorithm stops inside the theoretical region of convergence after $1000$ iterations to determine if it converges to the corresponding local minimum. Finally, we color the initialization points by the color of the corresponding local minimum PGD converges to (either green or red).
While Fig.~\ref{fig:example} verifies that our theoretical region of convergence falls inside the empirical region of convergence, it also reveals that our bound is conservative in this example.

\section{Conclusion and Future Work}
\label{sec:conc}

We introduced a novel analysis of linear convergence of projected gradient descent for the unit-modulus least squares problem. Our analysis reveals that near strict local minima, the convergence is linear as opposed to sublinear as suggested in \cite{tranter2017fast}. Moreover, we identified the sufficient conditions for linear convergence and provided an exact expression of the linear convergence rate. The theoretical rate predicts accurately the asymptotic convergence of PGD for UMLS in our numerical simulation. On the practical side, we propose two variant of PGD with adaptive step sizes that obtain fast convergence without prior knowledge about the solution.

For future work, we plan to improve our bound on the region of convergence. This requires further investigation into the bounding techniques used in the proof of Theorem~\ref{theo:rate}. Another potential direction is to develop the analysis for linear convergence of Bt-PGD and ARNAPGD. While convergence guarantees for backtracking line search and Nesterov's accelerated gradient have been proposed in the optimization literature \cite{boyd2004convex,nesterov2003introductory}, they often involve the spectral radius that depends linearly on the step size $\eta$. The UMLS problem, on the other hand, involve the spectral radius $\rho(\bm M_\eta)$ that depends non-linearly on $\eta$. This makes it challenging for determining closed-form expressions of the optimal step size in both plain PGD and accelerated PGD.

\appendices

\section{Proof of Lemma~\ref{lem:KKT2}}
\label{appdx:KKT2}

Since the constraint gradients are of form $\{\bm e_i \otimes \bm S_i (\bm x^*)\}_{i=1}^N$, the tangent space to $\C$ at $\bm x^*$ is given by
\begin{align*}
    T_{\bm x^*}\C = \Bigl\{ \bm y \in \R^{2N} \mid \bigl( \sum_{i=1}^N \bm e_i \bm e_i^\topnew \otimes \bm S_i (\bm x^*) \bigr)^\topnew \bm y = \bm 0_N \Bigr\} .
\end{align*}
Denote $\bm v_i = [-x^*_{2i}, x^*_{2i-1}]^\topnew$ for $i=1,\ldots,N$. A basis of $T_{\bm x^*}\C$ is given by $\{\bm e_i \otimes \bm v_i \}_{i=1}^N$, i.e., the columns of $\bm Z$. 
Alternatively, $T_{\bm x^*}\C$ can be represented as
\begin{align} \label{def:tangent}
    T_{\bm x^*}\C = \bigl\{ \bm Z \bm z \mid \bm z \in \R^N \bigr\} .
\end{align}

\noindent ($\Rightarrow$) From Chapter~11.5 in \cite{luenberger1984linear}, the second-order necessary condition for a stationary point $\bm x^*$ to be a local minimum of (\ref{prob:umls}) is $\bm y^\topnew \nabla^2_{\bm x} L(\bm x^*, \bm \gamma) \bm y \geq 0$ for all $\bm y \in T_{\bm x^*}\C$. In other words, for any $\bm z \in \R^N$, we have
\begin{align*}
    0 &\leq (\bm Z \bm z)^\topnew \bigl(\bm A^\topnew \bm A - \diag(\bm \gamma) \otimes \bm I_2 \bigr) (\bm Z \bm z) \\
    &= \bm z^\topnew \bigl(\bm Z^\topnew \bm A^\topnew \bm A \bm Z - \bm Z^\topnew (\diag(\bm \gamma) \otimes \bm I_2) \bm Z \bigr) \bm z \\
    &= \bm z^\topnew \bigl(\bm Z^\topnew \bm A^\topnew \bm A \bm Z - \bm Z^\topnew \bm Z \diag(\bm \gamma) \bigr) \bm z \\
    &= \bm z^\topnew \bigl( \bm Z^\topnew \bm A^\topnew \bm A \bm Z - \diag(\bm \gamma) \bigr) \bm z ,
\end{align*}
where the second equality stems from Lemma~\ref{lem:ZD} and the third equality uses the semi-orthogonality of $\bm Z$. Thus, we conclude that $\bm H \succeq \bm 0_N$.

\noindent ($\Leftarrow$) From Chapter~11.5 in \cite{luenberger1984linear}, the second-order sufficient condition for a stationary point $\bm x^*$ to be a local minimum of (\ref{prob:umls}) is $\bm y^\topnew \nabla^2_{\bm x} L(\bm x^*, \bm \gamma) \bm y > 0$ for all $\bm y \in T_{\bm x^*}\C$. By the same argument, this is equivalent to $\bm H \succ \bm 0_N$.

\section{Proof of Remark~\ref{rmk:H_Riemannian}}
\label{appdx:H_Riemannian}

Recall that the objective function is given by $f=\norm{\bm A \bm x - \bm b}^2/2$.
By definition of the Riemannian Hessian \cite{lee2018introduction}, for any vector fields $U,V : \C \to T\C$ on $\C$, we have
\begin{align} \label{def:hess_f}
    \text{Hess} f (U,V) = \langle \nabla_U \text{grad} f , V \rangle ,
\end{align}
where $\text{grad} f : \C \to T\C$ is the Riemannian gradient given by
\begin{align} \label{equ:grad_f}
    \text{grad} f (\bm x) = \bm Z \bm Z^\topnew \nabla f (\bm x) = \bm Z \bm Z^\topnew \bm A^\topnew (\bm A \bm x - \bm b) ,
\end{align}
for $\bm x \in \C$ and $\bm Z$ is the corresponding basis matrix of the tangent space to $\C$ at $\bm x$ (see Lemma~\ref{lem:KKT2}).
In addition, $\nabla_U \text{grad} f$ is the covariant derivative of the vector field $\text{grad} f$ in the direction of the vector field $U$.
It is fact that the covariant derivative is the orthogonal projection of the directional derivative onto the tangent space of the manifold, i.e.,
\begin{align} \label{equ:covar_grad_f}
    \nabla_U \text{grad} f (\bm x) &= \bm Z \bm Z^\topnew D_U \text{grad} f (\bm x) \nonumber \\
    &= \bm Z \bm Z^\topnew \lim_{t \to 0} \frac{\text{grad} f (\bm x + t \bm u) - \text{grad} f (\bm x)}{t} ,
\end{align}
where $\bm u = U(\bm x)$.
Substituting (\ref{equ:grad_f}) into the numerator on the RHS of (\ref{equ:covar_grad_f}) and simplifying the expression, we obtain
\begin{align*}
    &\nabla_U \text{grad} f (\bm x) = \bm Z \bm Z^\topnew \bigl( \bm A^\topnew \bm A \bm u - \bm B \bm A^\topnew (\bm A \bm x - \bm b) \bigr) ,
\end{align*}
where
\begin{align*}
    \bm B = \sum_{i=1}^N \bm e_i \bm e_i^\topnew \otimes \Bigl( \bm S_i(\bm u) \bigl(\bm S_i(\bm x) \bigr)^\topnew + \bm S_i(\bm x) \bigl(\bm S_i(\bm u) \bigr)^\topnew \Bigr) .
\end{align*}
Now, denoting $\bm v = V(\bm x)$ and evaluating (\ref{def:hess_f}) at $\bm x$ yields
\begin{align} \label{equ:hess_f_uv}
    \text{Hess} f_{\bm x} (\bm u, \bm v) &= \bm v^\topnew \bm Z \bm Z^\topnew \bigl( \bm A^\topnew \bm A \bm u - \bm B \bm A^\topnew (\bm A \bm x - \bm b) \bigr) \nonumber \\
    &= \bm v^\topnew \bigl( \bm A^\topnew \bm A \bm u - \bm B \bm A^\topnew (\bm A \bm x - \bm b) \bigr) ,
\end{align}
where the last equality stems from $\bm v \in T_{\bm x} \C$ and hence, $\bm v = \bm Z \bm Z^\topnew \bm v$.
In the case $\bm x = \bm x^*$ is a stationary point of (\ref{prob:umls}) with the Lagrange multiplier $\bm \gamma$, one can substituting (\ref{equ:KKT1}) into (\ref{equ:hess_f_uv}) to obtain
\begin{align} \label{equ:hess_f_uv2}
    \text{Hess} f_{\bm x} (\bm u, \bm v) &= \bm v^\topnew \bigl( \bm A^\topnew \bm A \bm u - \bm B (\diag(\bm \gamma) \otimes \bm I_2) \bm x \bigr) .
\end{align}
Notice that $\bm x = \sum_{i=1}^N \bm e_i \otimes \bm S_i(\bm x)$ and $\bigl(\bm S_i(\bm u) \bigr)^\topnew \bm S_i(\bm x) = 0$ for all $i=1,\ldots,N$. Therefore, the second term on the RHS of (\ref{equ:hess_f_uv2}) can be simplified as
\begin{align*}
    \bm B (\diag(\bm \gamma) \otimes \bm I_2) \bm x &= \sum_{i=1}^N \gamma_i \bm e_i \otimes \bm S_i(\bm u) \\
    &=(\diag(\bm \gamma) \otimes \bm I_2) \bm u .
\end{align*}
Substituting back into (\ref{equ:hess_f_uv2}) and reorganizing terms, we obtain the Riemannian Hessian as 
\begin{align} \label{equ:hess_f_uv_final}
    \text{Hess} f_{\bm x} (\bm u, \bm v) &= \bm u^\topnew \bigl( \bm A^\topnew \bm A - (\diag(\bm \gamma) \otimes \bm I_2) \bigr) \bm v .
\end{align}
Finally, it follows from (\ref{def:tangent}) that there is an one-to-one correspondence between the tangent space $T_{\bm x}\C$ and $\R^N$, i.e., $\bm u = \bm Z \tilde{\bm u}$ and $\bm v = \bm Z \tilde{\bm v}$ for $\tilde{\bm u}, \tilde{\bm v} \in \R^N$. 
Hence, we can define a bilinear function $H: \R^N \otimes \R^N \to \R$:
\begin{align*}
    H(\tilde{\bm u},\tilde{\bm v}) &\triangleq \text{Hess} f_{\bm x} (\bm u, \bm v) \\
    &= (\bm Z \tilde{\bm u})^\topnew \bigl( \bm A^\topnew \bm A - (\diag(\bm \gamma) \otimes \bm I_2) \bigr) (\bm Z \tilde{\bm v}) \\
    &= \tilde{\bm u}^\topnew (\bm Z^\topnew \bm A^\topnew \bm A \bm Z - \diag(\bm \gamma)) \tilde{\bm v} ,
\end{align*}
where the last equality stems from $\bm Z^\topnew \bm Z = \bm I_N$.
In other words, $\text{Hess} f_{\bm x}$ admits a compact matrix representation
\begin{align*}
    \bm H = \bm Z^\topnew \bm A^\topnew \bm A \bm Z - \diag(\bm \gamma) .
\end{align*}

\section{Proof of Lemma~\ref{lem:fixed}}
\label{appdx:fixed}

($\Rightarrow$) Assume $\bm x^*$ is a fixed point of Algorithm~\ref{algo:PGD} with step size $\eta>0$, i.e., 
\begin{align} \label{equ:xPcr}
    \bm x^* = \P_\C(\bm x^* - \eta \bm r) ,
\end{align}
where $\bm r = \bm A^\topnew (\bm A \bm x^* - \bm b)$.
We will show there exists $\bm \gamma \in \R^N$ such that for all $i=1,\ldots,N$, 
\begin{align} \label{equ:Sig}
    \bm S_i (\bm r) = \gamma_i \bm S_i (\bm x^*)
\end{align}
and
\begin{align} \label{equ:gamma_s}
    \begin{cases}
        \gamma_i < 1/\eta &\text{ if } \bm S_i(\bm x^*) \neq \bm s , \\
        \gamma_i \leq 1/\eta &\text{ if } \bm S_i(\bm x^*) = \bm s ,
    \end{cases}
\end{align}
where we recall that $\bm s = [1,0]^\topnew$.

For $i=1,\ldots,N$, applying the 2-selection operator $\bm S_i(\cdot)$ to both side of (\ref{equ:xPcr}) and substituting the RHS by the definition of $\P_\C$ in (\ref{equ:Pc}) yield
\begin{align} \label{equ:Sxs}
    \bm S_i (\bm x^*) = \begin{cases}
        \frac{\bm S_i (\bm x^* - \eta \bm r)}{\norm{\bm S_i (\bm x^* - \eta \bm r)}} &\text{ if } \bm S_i (\bm x^* - \eta \bm r) \neq \bm 0_2 , \\
        \bm s &\text{ if } \bm S_i (\bm x^* - \eta \bm r) = \bm 0_2 .
    \end{cases}
\end{align}
If $\bm S_i(\bm x^*) \neq \bm s$, then (\ref{equ:Sxs}) implies
\begin{align*}
    \bm S_i (\bm x^*) = \frac{\bm S_i (\bm x^* - \eta \bm r)}{\norm{\bm S_i (\bm x^* - \eta \bm r)}} = \frac{\bm S_i (\bm x^*) - \eta \bm S_i (\bm r)}{\norm{\bm S_i (\bm x^* - \eta \bm r)}} ,
\end{align*}
which in turns can be reorganized as $\bm S_i (\bm r) = \gamma_i \bm S_i (\bm x^*)$ for
\begin{align} \label{equ:gamma_s_le}
    \gamma_i = \frac{1-\norm{\bm S_i (\bm x^*) - \eta \bm S_i(\bm r)}}{\eta} < \frac{1}{\eta} .
\end{align}
If $\bm S_i(\bm x^*) = \bm s$, we consider two sub-cases:
\begin{enumerate}[leftmargin=*]
    \item If $\bm S_i (\bm x^* - \eta \bm r) \neq \bm 0_2$, then by the same argument as the previous case, we obtain (\ref{equ:gamma_s_le}).
    \item If $\bm S_i (\bm x^* - \eta \bm r) = \bm 0_2$, then using the linearity of $\bm S_i$, we have $\bm S_i (\bm r) = \gamma_i \bm S_i(\bm x^*)$ where $\gamma_i = 1/\eta$.
\end{enumerate}
In all cases, we have (\ref{equ:Sig}) and (\ref{equ:gamma_s}) hold. Finally, we note that the stationarity condition (\ref{equ:KKT1}) is equivalent to $\bm S_i (\bm r) = \gamma_i \bm S_i (\bm x^*)$ for all $i=1,\ldots,N$.

($\Leftarrow$) Assume $\bm x^*$ is a stationary point of (\ref{prob:umls}) (i.e., (\ref{equ:Sig} holds for all $i=1,\ldots,N$) with the corresponding Lagrange multiplier $\bm \gamma$ satisfying (\ref{equ:gamma_s}) for all $i=1,\ldots,N$. We will prove (\ref{equ:xPcr}) by showing that
\begin{align} \label{equ:SiPcx}
    \bm S_i \bigl( \P_\C (\bm x^* - \eta \bm r) \bigr) = \bm S_i(\bm x^*) ,
\end{align}
for any $i=1,\ldots,N$.

By the definition of $\P_\C$ in (\ref{equ:Pc}), we have
\begin{align} \label{equ:SiPc_2}
    \bm S_i \bigl( \P_\C (\bm x^* - \eta \bm r) \bigr) = \begin{cases}
        \frac{\bm S_i (\bm x^* - \eta \bm r)}{\norm{\bm S_i (\bm x^* - \eta \bm r)}} &\text{ if } \bm S_i (\bm x^* - \eta \bm r) \neq \bm 0_2 , \\
        \bm s &\text{ if } \bm S_i (\bm x^* - \eta \bm r) = \bm 0_2 .
    \end{cases}
\end{align}
Using the linearity of $\bm S_i(\cdot)$ and then the stationarity condition in (\ref{equ:Sig}) yield
\begin{align*} 
    \bm S_i(\bm x^* - \eta \bm r) &= \bm S_i(\bm x^*) - \eta \bm S_i (\bm r) \\
    &= \bm S_i(\bm x^*) - \eta \gamma_i \bm S_i(\bm x^*) = (1-\eta \gamma_i) \bm S_i(\bm x^*) . \numberthis \label{equ:Si_gammai}
\end{align*}
Since $\bm x \in \C$, $\norm{\bm S_i(\bm x^*)}=1$.
Taking the norm of both sides in (\ref{equ:Si_gammai}) and using (\ref{equ:gamma_s}) to remove the absolute value, we obtain
\begin{align*}
    \norm{\bm S_i(\bm x^* - \eta \bm r)} &= \norm{(1-\eta \gamma_i) \bm S_i(\bm x^*)} \\
    &= \abs{1-\eta \gamma} \norm{\bm S_i(\bm x^*)} = 1-\eta \gamma .
\end{align*}
Therefore, (\ref{equ:SiPc_2}) is equivalent to
\begin{align*} 
    \bm S_i \bigl( \P_\C (\bm x^* - \eta \bm r) \bigr) &= \begin{cases}
        \bm S_i(\bm x^*) &\text{ if } 1-\eta \gamma_i \neq 0 , \\
        \bm s &\text{ if } 1-\eta \gamma_i = 0 .
    \end{cases} \numberthis \label{equ:SiPc}
\end{align*}
    
\begin{itemize}[leftmargin=*]
    \item If $1-\eta \gamma_i \neq 0$, then (\ref{equ:SiPcx}) holds trivially.
    
    \item If $1-\eta \gamma_i = 0$, then $\bm S_i ( \P_\C (\bm x^* - \eta \bm r)) = \bm s$ and $\gamma_i = 1/\eta$. From (\ref{equ:gamma_s}), the latter only holds if $\bm S_i(\bm x^*) = \bm s$. Thus, we obtain $\bm S_i ( \P_\C (\bm x^* - \eta \bm r)) = \bm S_i(\bm x^*) = \bm s$.
\end{itemize}
In both case, we have (\ref{equ:SiPcx}) holds for all $i=1,\ldots,N$.
This completes our proof of the lemma.

\section{Proof of Proposition~\ref{prop:dPc}}
\label{appdx:dPc}

The proof of this lemma is based on the following result for the projection onto the unit sphere \cite{vu2021asymptotic}:
\begin{lemma} \label{lem:proj_unit}
(Rephrased from Lemma~5 in \cite{vu2021asymptotic})
Let $\bm x$ be a point on the unit sphere $\S^{n-1}$. Then, for any $\bm \delta \in \R^n$, the projection onto $\S^{n-1}$ satisfies
\begin{align} \label{equ:taylor_sphere}
  \P_{\S^{n-1}}(\bm x + \bm \delta) = \bm x + \bigl( \bm I - \bm x \bm x^\topnew \bigr) \bm \delta + \bm q_{\S^{n-1}} (\bm \delta) ,
\end{align}
where $\norm{\bm q_{\S^{n-1}}(\bm \delta)} \leq 2 \norm{\bm \delta}^2$.
\end{lemma}
\noindent Applying Lemma~\ref{equ:taylor_sphere} to the unit sphere $\S^1$ (i.e., $n=2$), we have, for each $i=1,\ldots,N$,
\begin{align*}
    \bm S_i \bigl( &\P_\C (\bm x + \bm \delta) \bigr) = \P_{\S^1} \bigl( \bm S_i (\bm x + \bm \delta) \bigr) \\
    &= \P_{\S^1} \bigl( \bm S_i (\bm x) + \bm S_i(\bm \delta) \bigr) \\
    &= \bm S_i (\bm x) + \bigl( \bm I_2 - \bm S_i (\bm x) (\bm S_i (\bm x))^\topnew \bigr) \bm S_i(\bm \delta) + \bm q_{\S^{1}}\bigl(\bm S_i(\bm \delta)\bigr) \\
    &= \bm S_i (\bm x) + \bm v_i \bm v_i^\topnew \bm S_i(\bm \delta) + \bm q_{\S^{1}}\bigl(\bm S_i(\bm \delta)\bigr) ,
\end{align*}
where $\bm v_i = [-x_{2i},x_{2i-1}]^\topnew$. Using the property of the 2-selection operator in (\ref{equ:e_S}), we further have
\begin{align*}
    \P_\C (\bm x + \bm \delta) &= \sum_{i=1}^N \bm e_i \otimes \bm S_i \bigl( \P_\C (\bm x + \bm \delta) \bigr) \\
    &= \sum_{i=1}^N \bm e_i \otimes \Bigl( \bm S_i (\bm x) + \bm v_i \bm v_i^\topnew \bm S_i(\bm \delta) + \bm q_{\S^{1}}\bigl(\bm S_i(\bm \delta)\bigr) \Bigr) \\
    &= \sum_{i=1}^N \bm e_i \otimes \bm S_i (\bm x) + \sum_{i=1}^N (\bm e_i \otimes \bm v_i \bm v_i^\topnew) \bm S_i(\bm \delta) \\
    &\qquad + \sum_{i=1}^N \bm e_i \otimes \bm q_{\S^{1}}\bigl(\bm S_i(\bm \delta)\bigr) \\
    &= \bm x + \bm Z \bm Z^\topnew \bm \delta + \bm q(\bm \delta) , \numberthis \label{equ:Zv}
\end{align*}
where $\bm q(\bm \delta)$ satisfies $\bm S_i(\bm q(\bm \delta)) = \bm q_{\S^{1}}(\bm S_i(\bm \delta))$ and
\begin{align*}
    \norm{\bm q (\bm \delta)}^2 &= \sum_{i=1}^N \norm{\bm S_i(\bm q(\bm \delta))}^2 = \sum_{i=1}^N \norm{\bm q_{\S^1} (\bm S_i (\bm \delta))}^2 \\
    &\leq \sum_{i=1}^N \bigl(2 \norm{\bm S_i (\bm \delta)}^2\bigr)^2 \leq \bigl(\sum_{i=1}^N 2 \norm{\bm S_i (\bm \delta)}^2\bigr)^2 \\
    &= 4 \bigl(\sum_{i=1}^N (\delta_{2i-1}^2+\delta_{2i}^2) \bigr)^2 = 4 \bigl(\sum_{j=1}^{2N} \delta_j^2 \bigr)^2 = 4 \norm{\bm \delta}^4 .
\end{align*}
This completes our proof of the lemma.

\section{More Details on Condition (C3')}
\label{appdx:C3}

In this section, we show that when Conditions (C1) and (C2) in Theorem~\ref{theo:rate} hold, Condition (C3'), i.e., \begin{align} \label{equ:C3_gamma}
    \eta (\lambda_1(\bm H) + 2 \gamma_i) < 2 ,
\end{align}
for all $i=1,\ldots,N$, is sufficient for Condition (C3). First, we prove that $\bm D_\eta = (\bm I_N - \eta \diag(\bm \gamma))^{-1}$ is PSD. Second, we show that all the eigenvalues of $\bm D_\eta \bm H$ lie between $0$ and $(1-\eta \gamma_i)^{-1} \lambda_1(\bm H)$ (exclusively). Third, we claim that the spectral radius of $\bm M_\eta = \bm I_N - \eta \bm D_\eta \bm H$ is strictly less than $1$.

In the first step, rearranging (\ref{equ:C3_gamma}), we obtain $\eta \lambda_1(\bm H) /2 < 1-\eta \gamma_i$. By Condition (C2), we have $\lambda_1(\bm H)>0$. Since $\eta>0$, it follows that $0 < \eta \lambda_1(\bm H) /2 < 1-\eta \gamma_i$. Thus, the diagonal matrix $\bm D_\eta$ has all positive entries and hence, is a PSD matrix. 
In the second step, we use the inequalities for the eigenvalues of the product of two PSD matrices in \cite{wang1992some} to obtain
\begin{align} \label{equ:wang}
    \lambda_i(\bm D_\eta) \lambda_N(\bm H) \leq \lambda_i(\bm D_\eta \bm H) \leq \lambda_i(\bm D_\eta) \lambda_1(\bm H) ,
\end{align}
for all $i=1,\ldots,N$.
Since both $\bm D_\eta$ and $\bm H$ are PSD, we can lower bound the eigenvalues of $\bm D_\eta \bm H$ by $\lambda_i(\bm D_\eta \bm H) \geq \lambda_i(\bm D_\eta) \lambda_N(\bm H) > 0$. On the other hand, substituting $\lambda_i(\bm D_\eta) = (1-\eta \gamma_i)^{-1}$ into the upper bound in (\ref{equ:wang}) yields $\lambda_i(\bm D_\eta \bm H) \leq (1-\eta \gamma_i)^{-1} \lambda_1(\bm H)$.
Finally, using the fact that $\lambda_i(\bm M_\eta) = 1 - \eta \lambda_i(\bm D_\eta \bm H)$ and $0 < \lambda_i(\bm D_\eta \bm H) \leq (1-\eta \gamma_i)^{-1} \lambda_1(\bm H)$, for all $i=1,\ldots,N$, we obtain
\begin{align*}
    1 - \frac{\eta}{1-\eta \gamma_i} \lambda_1(\bm H) \leq \lambda_i(\bm M_\eta) < 1 .
\end{align*}
Now, rearranging (\ref{equ:C3_gamma}) to obtain $ 1 - \frac{\eta}{1-\eta \gamma_i} \lambda_1(\bm H) > -1$, we have all the eigenvalues of $\bm M_\eta$ lie between $-1$ and $1$ (exclusively). Since the spectral radius is the maximum of the absolute values of these eigenvalues, we conclude that $\rho(\bm M_\eta)<1$. This completes our proof in this section.

\section{Proof of Lemma~\ref{lem:PD_fixed}}
\label{appdx:PD_fixed}

In the first part of this proof, we show that $\gamma_i<1/\eta$ for all $i=1,\ldots,N$.
From Condition (C2), we have $\bm D_\eta = (\bm I_N - \eta \diag(\bm \gamma))^{-1}$ is invertible and hence, the expression of $\bm M_\eta$ in (\ref{def:Mn}) is well-defined. 
In addition, from Condition (C1), $\bm H$ has a unique PD square root $\bm H^{1/2}$, with the inverse $\bm H^{-1/2}$. 
Thus, we have
\begin{align*}
    \bm H^{1/2} &\bm M_\eta \bm H^{-1/2} \\
    &= \bm H^{1/2} \biggl( \bm I_N - \eta \Bigl(\bm I_N - \eta \diag(\bm \gamma) \Bigr)^{-1} \bm H \biggr) \bm H^{-1/2} \\
    &= \bm I_N - \eta \bm H^{1/2} \bm D_\eta \bm H^{1/2} \triangleq \tilde{\bm M}_\eta .
\end{align*}
This shows that $\bm M_\eta$ and $\tilde{\bm M}_\eta$ are similar matrices with the same set of eigenvalues. Combining this with Condition (C3), we obtain $\rho(\bm M_\eta)= \rho(\tilde{\bm M}_\eta)<1$. 
Since $\tilde{\bm M}_\eta$ is symmetric, it then holds that
\begin{align*}
    \tilde{\bm M}_\eta = \bm I_N - \eta \bm H^{1/2} \bm D_\eta \bm H^{-1/2} \prec \bm I_N ,
\end{align*}
which in turn yields $\bm H^{1/2} \bm D_\eta \bm H^{1/2} \succ \bm 0_N$.
By the definition of PD matrices, for any vector $\bm u \in \R^N$, it holds that $\bm u^\topnew \bm H^{1/2} \bm D_\eta \bm H^{1/2} \bm u > 0$. Alternatively, we can write $\bm v^\topnew \bm D_\eta \bm v > 0$, where $\bm v = \bm H^{1/2} \bm u$. Notice that the mapping between $\bm u$ and $\bm v$ is bijection, which means $\bm v^\topnew \bm D_\eta \bm v > 0$ also holds for any $\bm v \in \R^N$. Consequently, $\bm D_\eta = \diag([(1-\eta \gamma_1)^{-1},\ldots,(1-\eta \gamma_N)^{-1}])$ must be a PD matrix. Equivalently, we have $\gamma_i<1/\eta$ for all $i=1,\ldots,N$.

For the second part of the proof, we note that $\gamma_i<1/\eta$, for all $i=1,\ldots,N$, are sufficient conditions for the Lagrange multiplier condition (\ref{equ:fixed}) in Lemma~\ref{lem:fixed}.
Since a strict local minimum is also a stationary point of (\ref{prob:umls}), $\bm x^*$ must be a fixed point of Algorithm~\ref{algo:PGD} with the given step size $\eta$.
This completes our proof of the lemma.

\section{Proof of Lemma~\ref{lem:delta1}}
\label{appdx:delta1}

Using the PGD update in (\ref{equ:pgd}) and rewriting $\bm x^{(k)} = \bm x^* + \bm \delta^{(k)})$, we derive a recursion on the error vector as follows 
\begin{align*}
    &\bm \delta^{(k+1)} = \bm x^{(k+1)} - \bm x^* \\
    &= \P_\C \Bigl( \bm x^{(k)} - \eta \bm A^\topnew \bigl( \bm A \bm x^{(k)} - \bm b \bigr) \Bigr) - \bm x^* \\
    &= \P_\C \Bigl( (\bm x^* + \bm \delta^{(k)}) - \eta \bm A^\topnew \bigl( \bm A (\bm x^* + \bm \delta^{(k)}) - \bm b \bigr) \Bigr) - \bm x^* \\
    &= \P_\C \Bigl( \bigl(\bm x^* - \eta \bm A^\topnew (\bm A \bm x^* - \bm b) \bigr) + (\bm I_{2N} - \eta \bm A^\topnew \bm A) \bm \delta^{(k)} \Bigr) - \bm x^* . \numberthis \label{equ:delta_k1_Pc}
\end{align*} 
Since $\bm x^*$ is a stationary point of (\ref{prob:umls}), we have $\bm A^\topnew (\bm A \bm x^* - \bm b) = (\diag(\bm \gamma) \otimes \bm I_2) \bm x^*$. Then, the first term inside the projection $\P_\C$ on the RHS of (\ref{equ:delta_k1_Pc}) can be represented as
\begin{align*}
    \bm x^* - \eta \bm A^\topnew (\bm A \bm x^* - \bm b) &= \bigl( \bm I_{2N} - \eta \diag(\bm \gamma) \otimes \bm I_2 \bigr) \bm x^* \\
    &= \Bigl( \bigl( \bm I_N - \eta \diag(\bm \gamma) \bigr) \otimes \bm I_2 \Bigr) \bm x^* \\
    &= (\bm D_\eta^{-1} \otimes \bm I_2) \bm x^* = (\bm D_\eta \otimes \bm I_2)^{-1} \bm x^* .
\end{align*}
where we recall that $\bm D_\eta = (\bm I_N - \eta \diag(\bm \gamma) )^{-1} \succ \bm 0_N$ by Lemma~\ref{lem:PD_fixed}. 
Thus, we rewrite (\ref{equ:delta_k1_Pc}) as
\begin{align*}
    \bm \delta^{(k+1)} &= \P_\C \Bigl( (\bm D_\eta \otimes \bm I_2)^{-1} \bm x^* +  (\bm I_{2N} - \eta \bm A^\topnew \bm A) \bm \delta^{(k)} \Bigr) - \bm x^* .
\end{align*}
Now let $\bm y = \bm x^* + (\bm D_\eta \otimes \bm I_2) (\bm I_{2N} - \eta \bm A^\topnew \bm A) \bm \delta^{(k)}$ and using the modulus scale-invariant property of the projection $\P_\C((\bm D_\eta \otimes \bm I_2)^{-1} \bm y) = \P_\C(\bm y)$, for $\bm D_\eta \succ \bm 0_N$, we further obtain
\begin{align} \label{equ:delta_Pc_x}
    \bm \delta^{(k+1)} &= \P_\C \Bigl( \bm x^* + (\bm D_\eta \otimes \bm I_2) (\bm I_{2N} - \eta \bm A^\topnew \bm A) \bm \delta^{(k)} \Bigr) - \bm x^* .
\end{align}
Finally, applying Proposition~\ref{prop:dPc} with the perturbation $\bm \delta = (\bm D_\eta \otimes \bm I_2) (\bm I_{2N} - \eta \bm A^\topnew \bm A) \bm \delta^{(k)}$ at $\bm x = \bm x^* \in \C$, we have
\begin{align*}
    \P_\C \Bigl( \bm x^* + &(\bm D_\eta \otimes \bm I_2) (\bm I_{2N} - \eta \bm A^\topnew \bm A) \bm \delta^{(k)} \Bigr) \\
    &= \bm x^* + \bm Z\bm Z^\topnew (\bm D_\eta \otimes \bm I_2) (\bm I_{2N} - \eta \bm A^\topnew \bm A) \bm \delta^{(k)} \\
    &\qquad \qquad + \bm q \bigl( (\bm D_\eta \otimes \bm I_2) (\bm I_{2N} - \eta \bm A^\topnew \bm A) \bm \delta^{(k)} \bigr)
\end{align*}
Substituting this back into (\ref{equ:delta_Pc_x}) yields (\ref{equ:delta1}). This completes the proof of the lemma.

\section{Proof of Lemma~\ref{lem:delta2}}
\label{appdx:delta2}

Since $\bm x^{(k)}$ lies in $\C$, we can represent the error vector as
\begin{align*}
    \bm \delta^{(k)} &= \bm x^{(k)} - \bm x^* \\
    &= \P_\C (\bm x^{(k)}) - \bm x^* \\
    &= \P_\C (\bm x^* + \bm \delta^{(k)}) - \bm x^* . \numberthis \label{equ:delta_Pc}
\end{align*}
Using Proposition~\ref{prop:dPc}, we have
\begin{align*}
    \P_\C (\bm x^* + \bm \delta^{(k)}) = \bm x^* + \bm Z \bm Z^\topnew \bm \delta^{(k)} + \bm q (\bm \delta^{(k)}) .
\end{align*}
Substituting this back into the RHS of (\ref{equ:delta_Pc}) yields
\begin{align*}
    \bm \delta^{(k)} = \bm Z \bm Z^\topnew \bm \delta^{(k)} + \bm q (\bm \delta^{(k)}) .
\end{align*}
This completes our proof of the lemma.

\section{Proof of Lemma~\ref{lem:delta3}}
\label{appdx:delta3}

Substituting (\ref{equ:delta_k}) back into the first term on the RHS of (\ref{equ:delta1}), we have
\begin{align*}
    \bm \delta^{(k+1)} = &\bm Z\bm Z^\topnew (\bm D_\eta \otimes \bm I_2) (\bm I_{2N} - \eta \bm A^\topnew \bm A) \bm Z\bm Z^\topnew \bm \delta^{(k)} \\
    &+ \bm Z\bm Z^\topnew (\bm D_\eta \otimes \bm I_2) (\bm I_{2N} - \eta \bm A^\topnew \bm A) \bm q (\bm \delta^{(k)}) \\
    &+ \bm q \bigl( (\bm D_\eta \otimes \bm I_2) (\bm I_{2N} - \eta \bm A^\topnew \bm A) \bm \delta^{(k)} \bigr) . \numberthis \label{equ:delta_ZD}
\end{align*}

\noindent From Lemma~\ref{lem:ZD} and the fact that $\bm Z^\topnew \bm Z= \bm I_N$, we can represent (\ref{equ:delta_ZD}) as
\begin{align*}
    \bm \delta^{(k+1)} &= \bm Z \bm D_\eta \bm Z^\topnew (\bm I_{2N} - \eta \bm A^\topnew \bm A) \bm Z\bm Z^\topnew \bm \delta^{(k)} + \hat{\bm q}(\bm \delta^{(k)}) \\
    &= \bm Z \bm D_\eta (\bm I_N - \eta \bm Z^\topnew \bm A^\topnew \bm A \bm Z ) \bm Z^\topnew \bm \delta^{(k)} + \hat{\bm q}(\bm \delta^{(k)}) , \numberthis \label{equ:delta_ZDA}
\end{align*}
where $\hat{\bm q}(\bm \delta) = \bm Z\bm Z^\topnew (\bm D_\eta \otimes \bm I_2) (\bm I_{2N} - \eta \bm A^\topnew \bm A) \bm q (\bm \delta) + \bm q \bigl( (\bm D_\eta \otimes \bm I_2) (\bm I_{2N} - \eta \bm A^\topnew \bm A) \bm \delta \bigr)$.
Recall that $\bm H = \bm Z^\topnew \bm A^\topnew \bm A \bm Z - \diag(\bm \gamma)$. Thus, (\ref{equ:delta_ZDA}) is equivalent to
\begin{align*}
    \bm \delta^{(k+1)} &= \bm Z \bm D_\eta (\bm I_N - \eta \diag(\bm \gamma) - \bm H) \bm Z^\topnew \bm \delta^{(k)} + \hat{\bm q}(\bm \delta^{(k)}) \\
    &= \bm Z (\bm I_N - \eta \bm D_\eta \bm H) \bm Z^\topnew \bm \delta^{(k)} + \hat{\bm q}(\bm \delta^{(k)}) .
\end{align*}
By the definition of $\bm M_\eta$ in (\ref{def:Mn}), the last equation is the same as (\ref{equ:H_delta}).

To bound the norm of $\hat{\bm q}(\bm \delta^{(k)})$, we use the triangle inequality and the product norm inequality as follows
\begin{align*}
    \norm{\hat{\bm q}(\bm \delta)} &\leq \norm{\bm Z\bm Z^\topnew (\bm D_\eta \otimes \bm I_2) (\bm I_{2N} - \eta \bm A^\topnew \bm A) \bm q (\bm \delta)} \\
    &\quad + \norm{\bm q \bigl( (\bm D_\eta \otimes \bm I_2) (\bm I_{2N} - \eta \bm A^\topnew \bm A) \bm \delta \bigr)} \\
    &\leq \norm{\bm Z\bm Z^\topnew}_2 \norm{(\bm D_\eta \otimes \bm I_2) (\bm I_{2N} - \eta \bm A^\topnew \bm A)}_2 \norm{\bm q (\bm \delta)} \\
    &\quad + \norm{\bm q \bigl( (\bm D_\eta \otimes \bm I_2) (\bm I_{2N} - \eta \bm A^\topnew \bm A) \bm \delta \bigr)} .
\end{align*}
Since $\norm{\bm q(\bm \delta)} \leq 2 \norm{\bm \delta}$ (see Proposition~\ref{prop:dPc}) and $c_\eta = \norm{(\bm D_\eta \otimes \bm I_2) (\bm I_{2N} - \eta \bm A^\topnew \bm A)}_2$, we further obtain
\begin{align*}
    \norm{\hat{\bm q}(\bm \delta)} &\leq \norm{\bm Z\bm Z^\topnew}_2 \cdot c_\eta \cdot 2 \norm{\bm \delta}^2 \\
    &\qquad + 2 \norm{(\bm D_\eta \otimes \bm I_2) (\bm I_{2N} - \eta \bm A^\topnew \bm A) \bm \delta}^2 \\
    &\leq 2 c_\eta \norm{\bm Z\bm Z^\topnew}_2 \norm{\bm \delta}^2 + 2 c_\eta^2 \norm{\bm \delta}^2 \\
    &\leq 2 c_\eta \norm{\bm \delta}^2 +  2 c_\eta^2 \norm{\bm \delta}^2 ,
\end{align*}
where the last inequality stems from $\norm{\bm Z\bm Z^\topnew}_2 \leq 1$ since $\bm Z\bm Z^\topnew$ is an orthogonal projection matrix.
This completes our proof of the lemma.

\section{Proof of Lemma~\ref{lem:rho2}}
\label{appdx:rho2}

The proof in this section relies on Lemmas~\ref{lem:delta3},~\ref{lem:scalar}, and \ref{lem:eig_Hn}.
Let $\tilde{\bm \delta}^{(k)} = (\bm D_\eta^{-1/2} \otimes \bm I_2) \bm \delta^{(k)}$. Left-multiplying both sides of (\ref{equ:H_delta}) with $(\bm D_\eta^{-1/2} \otimes \bm I_2)$, we have
\begin{align*} 
    \tilde{\bm \delta}^{(k+1)} &= (\bm D_\eta^{-1/2} \otimes \bm I_2) \bm Z \bm M_\eta \bm Z^\topnew \bm \delta^{(k)} + (\bm D_\eta^{-1/2} \otimes \bm I_2) \hat{\bm q}(\bm \delta^{(k)}) \\
    &= (\bm D_\eta^{-1/2} \otimes \bm I_2) \bm Z \bm M_\eta \bm Z^\topnew (\bm D_\eta^{1/2} \otimes \bm I_2) \tilde{\bm \delta}^{(k)} \\
    &\qquad + (\bm D_\eta^{-1/2} \otimes \bm I_2) \hat{\bm q}\bigl((\bm D_\eta^{1/2} \otimes \bm I_2) \tilde{\bm \delta}^{(k)} \bigr) . \numberthis \label{equ:DZHZ}
\end{align*}
Using Lemma~\ref{lem:ZD} and substituting $\bm M_\eta = \bm I_N - \eta \bm D_\eta^{-1} \bm H$ into the RHS of (\ref{equ:DZHZ}) yield
\begin{align*}
    \tilde{\bm \delta}^{(k+1)} &= \bm Z \bm D_\eta^{-1/2} (\bm I_N - \eta \bm D_\eta^{-1} \bm H) \bm D_\eta^{1/2} \bm Z^\topnew \tilde{\bm \delta}^{(k)} + \tilde{\bm q} (\tilde{\bm \delta}^{(k)}) \\
    &= \bm Z (\bm I_N - \eta \bm D_\eta^{-1/2} \bm H \bm D_\eta^{-1/2}) \bm Z^\topnew \tilde{\bm \delta}^{(k)} + \tilde{\bm q} (\tilde{\bm \delta}^{(k)}) , \numberthis \label{equ:ZDGDZ}
\end{align*}
where $\tilde{\bm q} (\tilde{\bm \delta}^{(k)}) = (\bm D_\eta^{-1/2} \otimes \bm I_2) \hat{\bm q}((\bm D_\eta^{1/2} \otimes \bm I_2) \tilde{\bm \delta}^{(k)})$ satisfies
\begin{align*}
    \norm{\tilde{\bm q} (\tilde{\bm \delta}^{(k)})} &\leq \norm{\bm D_\eta^{-1/2} \otimes \bm I_2}_2 \norm{\hat{\bm q}((\bm D_\eta^{1/2} \otimes \bm I_2) \tilde{\bm \delta}^{(k)})} \\
    &= \norm{\bm D_\eta^{-1/2}}_2 \norm{\hat{\bm q}\bigl((\bm D_\eta^{1/2} \otimes \bm I_2) \tilde{\bm \delta}^{(k)}\bigr)} \\
    &\leq \norm{\bm D_\eta^{-1/2}}_2 \cdot 2 c_\eta (c_\eta+1) \norm{(\bm D_\eta^{1/2} \otimes \bm I_2) \tilde{\bm \delta}^{(k)}}^2 \\
    &\leq 2 c_\eta (c_\eta+1) \norm{\bm D_\eta^{-1/2}}_2 \norm{\bm D_\eta^{1/2} \otimes \bm I_2}_2^2 \norm{\tilde{\bm \delta}^{(k)}}^2 \\
    &\leq 2 c_\eta (c_\eta+1) \norm{\bm D_\eta^{-1/2}}_2 \norm{\bm D_\eta^{1/2}}_2^2  \norm{\tilde{\bm \delta}^{(k)}}^2 \\
    &= 2 c_\eta (c_\eta+1) (1-\eta \underline{\gamma})^{1/2} (1-\eta \overline{\gamma})^{-1} \norm{\tilde{\bm \delta}^{(k)}}^2 ,
\end{align*}
where the last equality stems from $\norm{\bm D_\eta^{-1/2}}_2 = (1-\eta \underline{\gamma})^{1/2}$ and $\norm{\bm D_\eta^{1/2}}_2 = (1-\eta \overline{\gamma})^{-1/2}$.
Let $q = 2 c_\eta (c_\eta+1) (1-\eta \underline{\gamma})^{1/2} (1-\eta \overline{\gamma})^{-1}$.
Taking the norm of both sides of (\ref{equ:ZDGDZ}) and then using the triangle inequality on the RHS, we obtain
\begin{align*}
    &\norm{\tilde{\bm \delta}^{(k+1)}} = \norm{\bm Z (\bm I_N - \eta \bm D_\eta^{-1/2} \bm H \bm D_\eta^{-1/2}) \bm Z^\topnew \tilde{\bm \delta}^{(k)} + \tilde{\bm q} (\tilde{\bm \delta}^{(k)})} \\
    &\quad \leq \norm{\bm Z (\bm I_N - \eta \bm D_\eta^{-1/2} \bm H \bm D_\eta^{-1/2}) \bm Z^\topnew \tilde{\bm \delta}^{(k)}} + \norm{\tilde{\bm q} (\tilde{\bm \delta}^{(k)})} .
\end{align*}
Since $\bm Z (\bm I_N - \eta \bm D_\eta^{-1/2} \bm H \bm D_\eta^{-1/2}) \bm Z^\topnew$ is symmetric, its spectral norm equals to its spectral radius.
The last inequality can be rewritten as
\begin{align*}
    \norm{\tilde{\bm \delta}^{(k+1)}} \leq \rho \bigl(\bm Z (\bm I_N - \eta \bm D_\eta^{-1/2} \bm H &\bm D_\eta^{-1/2}) \bm Z^\topnew \bigr) \norm{\tilde{\bm \delta}^{(k)}} \\
    &+ q \norm{\tilde{\bm \delta}^{(k)}}^2 . \numberthis \label{equ:ZDGDZ_inequ}
\end{align*}
Moreover, it can be seen from (\ref{equ:ZDGDZ}) that
\begin{align*}
    \bm Z (\bm I_N &- \eta \bm D_\eta^{-1/2} \bm H \bm D_\eta^{-1/2}) \bm Z^\topnew \\
    &= (\bm D_\eta^{-1/2} \otimes \bm I_2) \bm Z \bm M_\eta \bm Z^\topnew \bigl(\bm D_\eta^{-1/2} \otimes \bm I_2\bigr)^{-1} ,
\end{align*}
which in turns implies the two matrices $\bm Z (\bm I_N - \eta \bm D_\eta^{-1/2} \bm H \bm D_\eta^{-1/2}) \bm Z^\topnew$ and $\bm Z \bm M_\eta \bm Z^\topnew$ are similar and have the same spectral radius. 
In particular, we have
\begin{align*}
    \rho \bigl(\bm Z (\bm I_N - \eta \bm D_\eta^{-1/2} \bm H \bm D_\eta^{-1/2}) \bm Z^\topnew \bigr) &= \rho (\bm Z \bm M_\eta \bm Z^\topnew) \\
    &= \rho(\bm M_\eta) ,
\end{align*}
where the second equality stems from Lemma~\ref{lem:eig_Hn}.
Thus, (\ref{equ:ZDGDZ_inequ}) can be represented as
\begin{align*}
    \norm{\tilde{\bm \delta}^{(k+1)}} &\leq \rho(\bm M_\eta) \norm{\tilde{\bm \delta}^{(k)}} + q \norm{\tilde{\bm \delta}^{(k)}}^2 .
\end{align*}
Applying Lemma~\ref{lem:scalar2} with $b_k = \norm{\tilde{\bm \delta}^{(k)}}$, $\rho=\rho(\bm M_\eta)$, and 
\begin{align*}
    c &= \frac{\rho(\bm M_\eta) \bigl(1-\rho(\bm M_\eta)\bigr)}{q} = (1-\eta \underline{\gamma})^{1/2} c_1(\bm x^*,\eta) ,
\end{align*}
it holds that if $\norm{\tilde{\bm \delta}^{(0)}} < c$, then
\begin{align} \label{equ:inequ_delta_tk}
    \norm{\tilde{\bm \delta}^{(k)}} \leq \biggl(1 - \frac{\norm{\tilde{\bm \delta}^{(0)}}}{c}\biggr)^{-1} \norm{\tilde{\bm \delta}^{(0)}} \rho^k(\bm M_\eta) .
\end{align}
Recall that $\bm \delta^{(k)} = (\bm D_\eta^{1/2} \otimes \bm I_2) \tilde{\bm \delta}^{(k)}$. On the one hand, the LHS of (\ref{equ:inequ_delta_tk}) can be lower-bounded by $(1-\eta \overline{\gamma})^{1/2} \norm{\bm \delta^{(k)}}$ since
\begin{align*}
    \norm{\bm \delta^{(k)}} &= \norm{(\bm D_\eta^{1/2} \otimes \bm I_2) \tilde{\bm \delta}^{(k)}} \leq \norm{\bm D_\eta^{1/2} \otimes \bm I_2}_2 \norm{\tilde{\bm \delta}^{(k)}} \\
    &= \norm{\bm D_\eta^{1/2}}_2 \norm{\tilde{\bm \delta}^{(k)}} = (1-\eta \overline{\gamma})^{-1/2} \norm{\tilde{\bm \delta}^{(k)}} .
\end{align*}
On the other hand, the RHS of (\ref{equ:inequ_delta_tk}) can be upper-bounded as follows. Since
\begin{align*}
    \norm{\tilde{\bm \delta}^{(0)}} &= \norm{(\bm D_\eta^{-1/2} \otimes \bm I_2) {\bm \delta}^{(0)}} \leq \norm{\bm D_\eta^{-1/2} \otimes \bm I_2}_2 \norm{{\bm \delta}^{(0)}} \\
    &= \norm{\bm D_\eta^{-1/2}}_2 \norm{{\bm \delta}^{(0)}} = (1-\eta \underline{\gamma})^{1/2} \norm{{\bm \delta}^{(0)}} , \numberthis \label{equ:delta_t_upperbound}
\end{align*}
we have
\begin{align*}
    \biggl(&1 - \frac{\norm{\tilde{\bm \delta}^{(0)}}}{c}\biggr)^{-1} \norm{\tilde{\bm \delta}^{(0)}} \rho^k(\bm M_\eta) \\
    &\leq \biggl(1 - \frac{(1-\eta \underline{\gamma})^{1/2} \norm{{\bm \delta}^{(0)}}}{c}\biggr)^{-1} (1-\eta \underline{\gamma})^{1/2} \norm{{\bm \delta}^{(0)}} \rho^k(\bm M_\eta) \\
    &= \biggl(1 - \frac{\norm{{\bm \delta}^{(0)}}}{c_1(\bm x^*,\eta)}\biggr)^{-1} (1-\eta \underline{\gamma})^{1/2} \norm{{\bm \delta}^{(0)}} \rho^k(\bm M_\eta) . \numberthis \label{equ:upper_norm_delta}
\end{align*}
From the lower bound $(1-\eta \overline{\gamma})^{1/2} \norm{\bm \delta^{(k)}}$ and the upper bound in (\ref{equ:upper_norm_delta}), we obtain (\ref{equ:rho1_norm_delta}).
Finally, the region of convergence $\norm{\bm \delta^{(0)}} < c_1(\bm x^*,\eta)$ is sufficient to guarantee that $\norm{\tilde{\bm \delta}^{(0)}} < c = (1-\eta \underline{\gamma})^{1/2} c_1(\bm x^*,\eta)$ due to (\ref{equ:delta_t_upperbound}).
This completes our proof of the lemma.

\section{Auxiliary Lemmas}


\begin{lemma} \label{lem:ZD}
Given a matrix $\bm Z \in \R^{2N \times N}$ as in (\ref{equ:Z}). Then for any diagonal matrix $\bm D \in \R^{N \times N}$, we have $(\bm D \otimes \bm I_2) \bm Z = \bm Z \bm D$.
\end{lemma}

\begin{proof}
Recall from (\ref{equ:Zv}) that $\bm Z = \sum_{i=1}^N \bm e_i \bm e_i^\topnew \otimes \bm v_i$, where $\bm v_i = [-x_{2i},x_{2i-1}]^\topnew$.
By representing $\bm D = \sum_{i=1}^N D_{ii} \bm e_i \bm e_i^\topnew$, we have
\begin{align*}
    (\bm D \otimes \bm I_2) \bm Z &= \Bigl( \bigl( \sum_{i=1}^N D_{ii} \bm e_i \bm e_i^\topnew \bigr) \otimes \bm I_2 \Bigr) \cdot \bigl( \sum_{j=1}^N \bm e_j \bm e_j^\topnew \otimes \bm v_j \bigr) \\
    &= \sum_{i=1}^N \sum_{j=1}^N \Bigl( \bigl( D_{ii} \bm e_i \bm e_i^\topnew \bigr) \cdot \bigl( \bm e_j \bm e_j^\topnew \bigr) \Bigr) \otimes (\bm I_2 \cdot \bm v_j) \\
    &= \sum_{i=1}^N \sum_{j=1}^N D_{ii} \bigl( (\bm e_i^\topnew \bm e_j) \cdot \bm e_i \bm e_j^\topnew \bigr) \otimes \bm v_j \\
    &= \sum_{i=1}^N D_{ii} ( \bm e_i \bm e_i^\topnew ) \otimes \bm v_i \\
    &= \sum_{i=1}^N \sum_{j=1}^N \Bigl( \bigl( \bm e_i \bm e_i^\topnew \bigr) \cdot \bigl( D_{jj} \bm e_j \bm e_j^\topnew \bigr) \Bigr) \otimes (\bm v_i \cdot 1) \\
    &= \Bigl( \sum_{i=1}^N \bm e_i \bm e_i^\topnew \otimes \bm v_i \Bigr) \cdot \Bigl( \bigl( \sum_{i=1}^N D_{jj} \bm e_j \bm e_j^\topnew \bigr) \otimes 1 \Bigr) \\
    &= \Bigl( \sum_{i=1}^N \bm e_i \bm e_i^\topnew \otimes \bm v_i \Bigr) \cdot \bigl( \sum_{i=1}^N D_{jj} \bm e_j \bm e_j^\topnew \bigr) \\
    &= \bm Z \bm D ,
\end{align*}
where it is noted that 
\begin{align*}
    \bm e_i^\topnew \bm e_j = \begin{cases}
        1 &\text{ if } i=j , \\
        0 &\text{ if } i\neq j .
    \end{cases}
\end{align*}
\end{proof}

\begin{lemma} \label{lem:eig_Hn}
For any eigenvalue $\lambda$ of $\bm Z \bm M_\eta \bm Z^\topnew$, either $\lambda=0$ or $\lambda$ is an eigenvalue of $\bm M_\eta$. Consequently, we have
\begin{align*}
    \rho(\bm Z \bm M_\eta \bm Z^\topnew) = \rho(\bm M_\eta) .
\end{align*}
\end{lemma}

\begin{proof}
Let $(\lambda, \bm u)$ be a pair of eigenvalue and eigenvector of $\bm Z \bm M_\eta \bm Z^\topnew$. Then, we have
\begin{align} \label{equ:eig_ZHZ}
    \bm Z \bm M_\eta \bm Z^\topnew \bm u = \lambda \bm u .
\end{align}
Left-multiplying both sides of (\ref{equ:eig_ZHZ}) by $\bm Z^\topnew$ and using the semi-orthogonality of $\bm Z$, we obtain $\bm M_\eta (\bm Z^\topnew \bm u) = \lambda (\bm Z^\topnew \bm u)$.
This means either $\bm Z^\topnew \bm u = \bm 0_N$ or $\bm Z^\topnew \bm u$ is an eigenvector of $\bm M_\eta$. In the former case, we have $\lambda=0$. In the latter case, we have $\lambda$ is an eigenvalue of $\bm M_\eta$. 
Finally, since the spectral radius is the maximum absolute value of all eigenvalues, it is trivial that $\rho(\bm Z \bm M_\eta \bm Z^\topnew) = \rho(\bm M_\eta)$.
\end{proof}

\begin{lemma} 
\label{lem:scalar}
(Rephrased from the supplemental material of \cite{vu2018adaptive}) Let $\{a_k\}_{k=0}^\infty \subset \R_+$ be the sequence defined by 
\begin{align} \label{equ:scalar}
a_{k+1} = \rho a_k + q a_k^2 \qquad \text{ for } k=0,1,\ldots ,
\end{align}
where $0<\rho<1$ and $q \ge 0$. Then $\{a_k\}_{k=0}^\infty$ converges \textbf{monotonically} to $0$ if and only if $a_0 < \frac{1-\rho}{q}$. 
A simple linear convergence bound can be derived for $a_0 < \rho \frac{1-\rho}{q}$ in the form of 
\begin{align} \label{equ:ak}
    a_k \leq \biggl(1 - \frac{a_0 q}{\rho(1-\rho)}\biggr)^{-1} a_0 \rho^k .
\end{align}
\end{lemma}

\begin{proof}
For each $k \in {\mathbb N}$, let us define $d_k = {a_k}/(a_0 \rho^k)$.
Substituting $a_k = a_0 d_k \rho^k$ into (\ref{equ:scalar}) and defining $\tau = a_0 q / (1-\rho)$, we obtain
\begin{align*}
\begin{cases}
    d_0 = 1, \\
    d_{k+1} = d_k+\tau (1-\rho)\rho^{k-1}d_k^2 \quad \text{for } k=0,1,\ldots .
\end{cases}
\end{align*}
Since $\tau (1-\rho)\rho^{k-1}d_k^2 > 0$, the sequence $\{d_k\}_{k=0}^\infty$ is strictly increasing and positive. Thus, using $d_i > d_{i+1} > 0$, for any $i = 0,1,\ldots,k-1$, we have
\begin{align*}
    \frac{1}{d_i}-\frac{1}{d_{i+1}} &= \frac{d_{i+1}-d_i}{d_{i+1}d_i} < \frac{d_{i+1}-d_i}{d_i^2} = \tau (1-\rho)\rho^{i-1} .
\end{align*}
Summing over $i=0,1,\ldots,k-1$, we obtain
\begin{align*}
    1 - \frac{1}{d_k} &< \sum_{i=0}^{k-1} \tau (1-\rho) \rho^{i-1} = \frac{\tau}{\rho} (1-\rho^k) < \frac{\tau}{\rho} . \numberthis \label{equ:dk1}
\end{align*}
Substituting $d_k = {a_k}/(a_0 \rho^k)$ and $\tau = a_0 q / (1-\rho)$ into (\ref{equ:dk1}) and rearranging terms yield the desired bound on $a_k$ in (\ref{equ:ak}).
\end{proof}

\begin{lemma} 
\label{lem:scalar2}
Let $\{b_k\}_{k=0}^\infty \subset \R_+$ be the sequence defined by 
\begin{align} \label{equ:scalar2}
b_{k+1} \leq \rho b_k + q b_k^2 \qquad \text{ for } k=0,1,\ldots ,
\end{align}
where $0<\rho<1$ and $q \ge 0$. 
If $b_0 < \frac{1-\rho}{q}$, then $\{b_k\}_{k=0}^\infty$ converges to $0$.
If $b_0 < c \triangleq \rho \frac{1-\rho}{q}$, then for any integer $k \geq 0$,
\begin{align*}
    b_k \leq \biggl(1 - \frac{b_0 }{c}\biggr)^{-1} b_0 \rho^k .
\end{align*}
\end{lemma}

\begin{proof}
Let us define a surrogate sequence $\{a_k\}_{k=0}^\infty$ that upper-bounds $\{b_k\}_{k=0}^\infty$ as follows
\begin{align*}
    \begin{cases}
        a_0 = b_0 , \\
        a_{k+1} = \rho a_k + q a_k^2 .
    \end{cases}
\end{align*}
First, we prove by induction that 
\begin{align} \label{equ:bk_ak}
    b_k \leq a_k \quad \forall k \in \mathbb{N}  .
\end{align}
The base case when $k=0$ holds trivially as $b_0=a_0$.
In the induction step, given $b_k \leq a_k$ for an integer $k \geq 0$, we have
\begin{align*}
    b_{k+1} &\leq \rho b_k + q b_k^2 \leq \rho a_k + a_k^2 = a_{k+1} .
\end{align*}
By the principle of induction, (\ref{equ:bk_ak}) holds for all $k \in \mathbb{N}$. 
Now, by Lemma~\ref{lem:scalar}, we have
\begin{align*}
    b_k \leq a_k &\leq \biggl(1 - \frac{a_0 q}{\rho(1-\rho)}\biggr)^{-1} a_0 \rho^k \\
    &= \biggl(1 - \frac{b_0 q}{\rho(1-\rho)}\biggr)^{-1} b_0 \rho^k .
\end{align*}
This completes our proof of the lemma.
\end{proof}

\ifCLASSOPTIONcaptionsoff
  \newpage
\fi




%
\bibliographystyle{IEEEtran}
\bibliography{IEEEabrv,refs}

\begin{thebibliography}{10}
\providecommand{\url}[1]{#1}
\csname url@samestyle\endcsname
\providecommand{\newblock}{\relax}
\providecommand{\bibinfo}[2]{#2}
\providecommand{\BIBentrySTDinterwordspacing}{\spaceskip=0pt\relax}
\providecommand{\BIBentryALTinterwordstretchfactor}{4}
\providecommand{\BIBentryALTinterwordspacing}{\spaceskip=\fontdimen2\font plus
\BIBentryALTinterwordstretchfactor\fontdimen3\font minus
  \fontdimen4\font\relax}
\providecommand{\BIBforeignlanguage}[2]{{%
\expandafter\ifx\csname l@#1\endcsname\relax
\typeout{** WARNING: IEEEtran.bst: No hyphenation pattern has been}%
\typeout{** loaded for the language `#1'. Using the pattern for}%
\typeout{** the default language instead.}%
\else
\language=\csname l@#1\endcsname
\fi
#2}}
\providecommand{\BIBdecl}{\relax}
\BIBdecl

\bibitem{lu2013novel}
C.-J. Lu, W.-X. Sheng, Y.-B. Han, and X.-F. Ma, ``A novel adaptive phase-only
  beamforming algorithm based on semidefinite relaxation,'' in \emph{{IEEE}
  Int. Symp. Phased Array Syst. Technol.}\hskip 1em plus 0.5em minus
  0.4em\relax IEEE, 2013, pp. 617--621.

\bibitem{tranter2017fast}
J.~Tranter, N.~D. Sidiropoulos, X.~Fu, and A.~Swami, ``Fast unit-modulus least
  squares with applications in beamforming,'' \emph{{IEEE} Trans. Signal
  Process.}, vol.~65, no.~11, pp. 2875--2887, 2017.

\bibitem{candes2013phaselift}
E.~J. Candes, T.~Strohmer, and V.~Voroninski, ``{PhaseLift}: Exact and stable
  signal recovery from magnitude measurements via convex programming,''
  \emph{Comm. Pure Appl. Math.}, vol.~66, no.~8, pp. 1241--1274, 2013.

\bibitem{waldspurger2015phase}
I.~Waldspurger, A.~d’Aspremont, and S.~Mallat, ``Phase recovery, {MaxCut} and
  complex semidefinite programming,'' \emph{Math. Program.}, vol. 149, no.~1,
  pp. 47--81, 2015.

\bibitem{thompson1976adaptation}
P.~Thompson, ``Adaptation by direct phase-shift adjustment in narrow-band
  adaptive antenna systems,'' \emph{{IEEE} Trans. Antennas Propag.}, vol.~24,
  no.~5, pp. 756--760, 1976.

\bibitem{soltanalian2014designing}
M.~Soltanalian and P.~Stoica, ``Designing unimodular codes via quadratic
  optimization,'' \emph{{IEEE} Trans. Signal Process.}, vol.~62, no.~5, pp.
  1221--1234, 2014.

\bibitem{fu2012complex}
X.~Fu, F.~K. Chan, W.-K. Ma, and H.~So, ``A complex-valued semidefinite
  relaxation approach for two-dimensional source localization using distance
  measurements and imperfect receiver positions,'' in \emph{{IEEE} Int. Conf.
  Signal Process.}, vol.~2.\hskip 1em plus 0.5em minus 0.4em\relax IEEE, 2012,
  pp. 1491--1494.

\bibitem{zhang2006complex}
S.~Zhang and Y.~Huang, ``Complex quadratic optimization and semidefinite
  programming,'' \emph{{SIAM} J. Optim.}, vol.~16, no.~3, pp. 871--890, 2006.

\bibitem{luo2010semidefinite}
Z.-Q. Luo, W.-K. Ma, A.~M.-C. So, Y.~Ye, and S.~Zhang, ``Semidefinite
  relaxation of quadratic optimization problems,'' \emph{{IEEE} Signal Process.
  Mag.}, vol.~27, no.~3, pp. 20--34, 2010.

\bibitem{nesterov2003introductory}
Y.~Nesterov, \emph{Introductory lectures on convex optimization: A basic
  course}.\hskip 1em plus 0.5em minus 0.4em\relax Springer Science \& Business
  Media, 2003, vol.~87.

\bibitem{vu2019convergence}
T.~Vu, R.~Raich, and X.~Fu, ``On convergence of projected gradient descent for
  minimizing a large-scale quadratic over the unit sphere,'' in \emph{Proc.
  {IEEE} Int. Workshop Mach. Learn. Signal Process.}\hskip 1em plus 0.5em minus
  0.4em\relax IEEE, 2019, pp. 1--6.

\bibitem{meyer2000matrix}
C.~D. Meyer, \emph{Matrix analysis and applied linear algebra}.\hskip 1em plus
  0.5em minus 0.4em\relax SIAM, 2000, vol.~71.

\bibitem{vasilyev2013depth}
F.~Vasilyev and A.~Y. Ivanitskiy, \emph{In-depth analysis of linear
  programming}.\hskip 1em plus 0.5em minus 0.4em\relax Springer Science \&
  Business Media, 2013.

\bibitem{beck2017first}
A.~Beck, \emph{First-order methods in optimization}.\hskip 1em plus 0.5em minus
  0.4em\relax SIAM, 2017.

\bibitem{sorensen1997minimization}
D.~C. Sorensen, ``Minimization of a large-scale quadratic function subject to a
  spherical constraint,'' \emph{{SIAM} J. Optim.}, vol.~7, no.~1, pp. 141--161,
  1997.

\bibitem{hager2001minimizing}
W.~W. Hager, ``Minimizing a quadratic over a sphere,'' \emph{{SIAM} J. Optim.},
  vol.~12, no.~1, pp. 188--208, 2001.

\bibitem{beck2018globally}
A.~Beck and Y.~Vaisbourd, ``Globally solving the trust region subproblem using
  simple first-order methods,'' \emph{{SIAM} J. Optim.}, vol.~28, no.~3, pp.
  1951--1967, 2018.

\bibitem{luenberger1984linear}
D.~G. Luenberger, Y.~Ye \emph{et~al.}, \emph{Linear and nonlinear
  programming}.\hskip 1em plus 0.5em minus 0.4em\relax Springer, 1984, vol.~2.

\bibitem{lee2018introduction}
J.~M. Lee, \emph{Introduction to Riemannian manifolds}.\hskip 1em plus 0.5em
  minus 0.4em\relax Springer, 2018.

\bibitem{bertsekas1997nonlinear}
D.~P. Bertsekas, ``Nonlinear programming,'' \emph{J. Oper. Res. Soc.}, vol.~48,
  no.~3, pp. 334--334, 1997.

\bibitem{lewis2008alternating}
A.~S. Lewis and J.~Malick, ``Alternating projections on manifolds,''
  \emph{Math. Oper. Res.}, vol.~33, no.~1, pp. 216--234, 2008.

\bibitem{jorge2006numerical}
N.~Jorge and J.~W. Stephen, \emph{Numerical optimization}.\hskip 1em plus 0.5em
  minus 0.4em\relax Spinger, 2006.

\bibitem{vu2021closed}
T.~Vu and R.~Raich, ``A closed-form bound on the asymptotic linear convergence
  of iterative methods via fixed point analysis,'' \emph{Optim. Lett.}, 2022.

\bibitem{bellman2008stability}
R.~Bellman, \emph{Stability theory of differential equations}.\hskip 1em plus
  0.5em minus 0.4em\relax Courier Corporation, 2008.

\bibitem{polyak1964some}
B.~T. Polyak, ``Some methods of speeding up the convergence of iteration
  methods,'' \emph{{USSR} Comput. Math. Math. Phys.}, vol.~4, no.~5, pp. 1--17,
  1964.

\bibitem{vu2021exact}
T.~Vu and R.~Raich, ``Exact linear convergence rate analysis for low-rank
  symmetric matrix completion via gradient descent,'' in \emph{Proc. {IEEE}
  Int. Conf. Acoust. Speech Signal Process.}\hskip 1em plus 0.5em minus
  0.4em\relax IEEE, 2021, pp. 3240--3244.

\bibitem{boyd2004convex}
S.~Boyd and L.~Vandenberghe, \emph{Convex optimization}.\hskip 1em plus 0.5em
  minus 0.4em\relax Cambridge university press, 2004.

\bibitem{lessard2016analysis}
L.~Lessard, B.~Recht, and A.~Packard, ``Analysis and design of optimization
  algorithms via integral quadratic constraints,'' \emph{{SIAM} J. Optim.},
  vol.~26, no.~1, pp. 57--95, 2016.

\bibitem{o2015adaptive}
B.~O’donoghue and E.~Candes, ``Adaptive restart for accelerated gradient
  schemes,'' \emph{Found. Comput. Math.}, vol.~15, no.~3, pp. 715--732, 2015.

\bibitem{vu2021asymptotic}
T.~Vu and R.~Raich, ``On asymptotic linear convergence of projected gradient
  descent for constrained least squares,'' \emph{arXiv preprint
  arXiv:2112.11760}, 2021.

\bibitem{wang1992some}
B.~Wang and F.~Zhang, ``Some inequalities for the eigenvalues of the product of
  positive semidefinite hermitian matrices,'' \emph{Linear Algebra Appl.}, vol.
  160, pp. 113--118, 1992.

\bibitem{vu2018adaptive}
T.~Vu and R.~Raich, ``Adaptive step size momentum method for deconvolution,''
  in \emph{Proc. {IEEE} Stat. Signal Process. Workshop}, 2018, pp. 438--442.

\end{thebibliography}

\vfill


\end{document}